\documentclass[12pt]{amsart}
\usepackage{amsmath}
\usepackage{amsfonts}
\usepackage{amssymb}
\usepackage[all,cmtip]{xy}           

\usepackage{bbm}
\usepackage{bbding}
\usepackage{txfonts}
\usepackage[shortlabels]{enumitem}
\usepackage{ifpdf}
\ifpdf
\usepackage[colorlinks,final,backref=page,hyperindex]{hyperref}
\else
\usepackage[colorlinks,final,backref=page,hyperindex,hypertex]{hyperref}
\fi
\usepackage{tikz-cd}
\usepackage[active]{srcltx}
\usepackage{tikz}
\usepackage{amscd}
\usepackage{bm}

\makeatletter

\newtheorem{thm}{Theorem}[section]
\newtheorem{prop}[thm]{Proposition}

\newtheorem{prop-def}{thm}[section]

\theoremstyle{definition}
\newtheorem{defn}[thm]{Definition}

\newtheorem{remark}[thm]{Remark}
\newtheorem{exam}[thm]{Example}

\newcommand{\nc}{\newcommand}

 \nc{\mbibitem}[1]{\bibitem{#1}} 
 \nc{\mrm}[1]{{\rm #1}}

\nc{\pl}{\cdot}
 \nc{\la}{\longrightarrow}
\nc{\ot}{\otimes}\nc{\ad}{{\mathrm{ad}}}
 \nc{\rar}{\rightarrow}
 \nc{\red}{\color{red}}

\nc{\PLA}{{\mathrm{3\text{-}Lie}}}
\nc{\bfk}{{\bf k}}
\nc{\C}{{\mathrm{C}}}
\nc{\RBA}{\mathsf{RB3\text{-}Lie^\lambda}}
\nc{\RBO}{\mathsf{RBO^\lambda}}
\nc{\End}{\mrm{End}}
\nc{\Ext}{\mrm{Ext}}
\nc{\Fil}{\mrm{Fil}}
\nc{\Fr}{\mrm{Fr}}
\nc{\Frob}{\mrm{Frob}}
\nc{\Gal}{\mrm{Gal}}
\nc{\GL}{\mrm{GL}}
\nc{\Hom}{\mrm{Hom}}
\nc{\Hoch}{\mrm{Hoch}}
\nc{\hsr}{\mrm{H}}
\nc{\hpol}{\mrm{HP}}
\nc{\im}{\mrm{im}}
\nc{\id}{\mrm{Id}}
\nc{\Id}{\mrm{Id}}
\nc{\Irr}{\mrm{Irr}}
\nc{\incl}{\mrm{incl}}
\nc{\length}{\mrm{length}}
\nc{\NLSW}{\mrm{NLSW}}
\nc{\Lie}{\mrm{Lie}}
\nc{\mchar}{\rm char}
\nc{\mpart}{\mrm{part}}
\nc{\ql}{{\QQ_\ell}}
\nc{\qp}{{\QQ_p}}
\nc{\rank}{\mrm{rank}}
\nc{\rcot}{\mrm{cot}}
\nc{\rdef}{\mrm{def}}
\nc{\rdiv}{{\rm div}}
\nc{\rmH}{ {\mathrm{H}}}
\nc{\rtf}{{\rm tf}}
\nc{\rtor}{{\rm tor}}
\nc{\res}{\mrm{res}}
\nc{\Sh}{{\mathrm{Sh}}}
\nc{\SL}{\mrm{SL}}
\nc{\Spec}{\mrm{Spec}}
\nc{\sgn}{{\mathrm{sgn}}}
\nc{\tor}{\mrm{tor}}
\nc{\Tr}{\mrm{Tr}}
\nc{\tr}{\mrm{tr}}
\nc{\wt}{\mrm{wt}}
\nc{\op}{\mrm{op}}
\nc{\tLie}{{\mrm{3-Lie}}}

\nc{\BA}{{\mathbb A}}   \nc{\CC}{{\mathbb C}}
\nc{\DD}{{\mathbb D}}   \nc{\EE}{{\mathbb E}}
\nc{\FF}{{\mathbb F}}   \nc{\GG}{{\mathbb G}}
\nc{\HH}{ \mathrm{HH}}   \nc{\LL}{{\mathbb L}}
\nc{\NN}{{\mathbb N}}   \nc{\PP}{{\mathbb P}}
\nc{\QQ}{{\mathbb Q}}   \nc{\RR}{{\mathbb R}}
\nc{\TT}{{\mathbb T}}   \nc{\VV}{{\mathbb V}}
\nc{\ZZ}{{\mathbb Z}}   \nc{\TP}{\widetilde{P}}
\nc{\m}{{\mathbbm m}}

\nc{\cala}{{\mathcal A}}    \nc{\calc}{{\mathcal C}}
\nc{\cald}{\mathcal{D}}     \nc{\cale}{{\mathcal E}}
\nc{\calf}{{\mathcal F}}    \nc{\calg}{{\mathcal G}}
\nc{\calh}{{\mathcal H}}    \nc{\cali}{{\mathcal I}}
\nc{\call}{{\mathcal L}}    \nc{\calm}{{\mathcal M}}
\nc{\caln}{{\mathcal N}}    \nc{\calo}{{\mathcal O}}
\nc{\calp}{{\mathcal P}}    \nc{\calr}{{\mathcal R}}
\nc{\cals}{{\mathcal S}}    \nc{\calt}{{\Omega}}
\nc{\calv}{{\mathcal V}}    \nc{\calw}{{\mathcal W}}
\nc{\calx}{{\mathcal X}}

\nc{\fraka}{{\mathfrak a}}
\nc{\frakb}{\mathfrak{b}}
\nc{\frakg}{{\frak g}}
\nc{\frakl}{{\frak l}}
\nc{\fraks}{{\frak s}}

\nc{\frakm}{{\frak m}}
\nc{\frakM}{{\frak M}}
\nc{\frakp}{{\frak p}}
\nc{\frakW}{{\frak W}}
\nc{\frakX}{{\frak X}}
\nc{\frakS}{{\frak S}}
\nc{\frakA}{{\frak A}}
\nc{\frakC}{{\frak{C}}}
\nc{\frakx}{{\frakx}}

\begin{document}

\title[ Rota-Baxter $3$-Lie algebras of  arbitrary weight]{Deformations and   cohomology theory of  Rota-Baxter $3$-Lie algebras of  arbitrary weights}

\author{Shuangjian Guo, Yufei Qin, $\mbox{Kai\ Wang}^\dagger$  and Guodong Zhou}
\address{Shuangjian Guo, School of Mathematics and Statistics, Guizhou University of Finance and Economics, Guiyang 550025, China}
\email{shuangjianguo@126.com }

\address{Yufei Qin, Kai Wang and Guodong Zhou,
  School of Mathematical Sciences, Shanghai Key laboratory of PMMP,
  East China Normal University,
 Shanghai 200241,
   China}
   \email{290673049@qq.com}
   \email{wangkai@math.ecnu.edu.cn }
\email{gdzhou@math.ecnu.edu.cn}

\date{\today}

\begin{abstract}
 A cohomology theory of  Rota-Baxter $3$-Lie algebras of  arbitrary weights is introduced. Formal deformations, abelian extensions,   skeletal  Rota-Baxter $3$-Lie  2-algebras   and  crossed modules   of  Rota-Baxter 3-Lie algebras   are interpreted by using
 lower degree cohomology groups.  
\end{abstract}

\keywords{abelian extension, cohomology, crossed module,  formal deformation, $3$-Lie algebra,  Rota-Baxter operator, skeletal  Rota-Baxter $3$-Lie  2-algebra of  arbitrary weights}

\subjclass[2010]{
16E40   
16S80   
12H05   
16S70   
}

\maketitle

\let\thefootnote\relax\footnotetext{$\dagger$Corresponding author: Kai Wang, E-mail: wangkai@math.ecnu.edu.cn .}

\tableofcontents

\allowdisplaybreaks

\section*{Introduction}

Rota-Baxter operators on associative algebras were introduced in 1960 by Baxter \cite{Bax60} in his study of the fluctuation theory in probability.  They  have also
been studied in connection with many areas of mathematics and physics, including combinatorics,
number theory, operads and quantum field theory.     Bai, Guo and Ni \cite{Bai07, BGN11, BGN10} introduced  Rota-Baxter operators on   Lie algebras in  order to study the classical Yang-Baxter equations and contribute to the study of integrable systems.    

Algebraic deformation theory originated in the work of Gerstenhaber \cite{Ger63, Ger64} on deformations of associative algebras and this theory is  related to the cohomology theory of associative algebras: Hochschild cohomology \cite{Hoc46}.
Recently,  Tang, Bai, Guo and Sheng \cite{TBGS19} developed deformation theory  and cohomology theory of $\mathcal{O}$-operators (also called relative Rota-Baxter operators of weight zero) on Lie algebras. Later, Lazarev, Sheng and Tang \cite{LST21, LST20} studied deformation theory and cohomology theory of relative Rota-Baxter Lie algebras of weight zero and found applications to triangular Lie bialgebras. They further  determined the $L_{\infty}$-algebra
 that controls deformations of a relative Rota-Baxter Lie algebra and introduced the notion of a homotopy relative Rota-Baxter Lie algebra of weight zero.
  Das and Misha \cite{DM20}  developed the corresponding weight zero cohomology theory for relative Rota-Baxter associative  algebras  and Das also introduced cohomology theory for relative Rota-Baxter associative  algebras and relative Rota-Baxter Lie algebras of arbitrary weights \cite{Das21a, Das21b}.
  Wang and Zhou \cite{WZ21,WZ22} defined cohomology theory for  (absolute) Rota-Baxter associative  algebras of arbitrary weights, determined the  underlying  $L_{\infty}$-algebra and  showed that the dg operad of homotopy Rota-Baxter  associative   algebras is the minimal model of that of Rota-Baxter  associative   algebras.



The concept of $n$-Lie algebras was  introduced by Filippov in \cite{Fil85}, which is closely related to many fields in mathematics and mathematical
physics.  In particular,  3-Lie algebras play an important role to  study  the supersymmetry and gauge symmetry transformations of
the world-volume theory of multiple coincident M2-branes.  Bai, Guo, Li and Wu  \cite{BGLW13} introduced Rota-Baxter $3$-Lie algebras  of  arbitrary weights    and showed
that they can be derived from Rota-Baxter Lie algebras and pre-Lie algebras and from Rota-Baxter commutative associative algebras with derivations.  Bai, Guo and Sheng   \cite{BGS19} introduced  the notion of a relative Rota-Baxter operator (also called an $\mathcal{O}$-operator) on a $3$-Lie algebra with respect to a representation  in the  study of solutions of $3$-Lie classical Yang-Baxter equation. Recently,  Tang, Hou and Sheng \cite{THS21} constructed a Lie 3-algebra, whose Maurer-Cartan elements are exactly relative Rota-Baxter operators on a $3$-Lie algebra,  introduced the cohomology theory of relative Rota-Baxter operators on a $3$-Lie algebra  and  studied deformations of relative Rota-Baxter operators by using their cohomology theory.

It is, therefore,  time to develop a cohomology theory of  (absolute) Rota-Baxter $3$-Lie algebras, by which we study formal deformations and abelian extensions of  Rota-Baxter $3$-Lie algebras.  It needs to be emphasized that
there are results on $3$-Lie algebras in this paper which are not ``parallel" to the case of Lie algebras given in \cite{Das21b}. Because of the complexity  of $3$-Lie algebras, we need some technique to complete the paper.

 It might be inspiring to explain our method. As deformation cohomology  control deformations, we first consider  formal deformations of Rota-Baxter $3$-Lie algebras, whose deformation equations should provide cohomology classes, thus giving hints for the general construction of the cochain complex of deformation cohomology, although some more effort was needed to guess the general form of deformation cohomology.
Although we present the definitions of our   cohomology theory in Section 3, this theory was indeed inspired from   deformation theory displayed  in Section 4.

 In a forthcoming paper, we will determine the $L_\infty$-algebra controlling deformations of Rota-Baxter $3$-Lie algebras, introduce homotopy Rota-Baxter $3$-Lie algebras and  provide  the minimal model of the operad of Rota-Baxter $3$-Lie algebras, thus this paper can be considered the elementary part of our project on   deformation theory and homotopy theory of  Rota-Baxter $3$-Lie algebras. In particular, we will see that the $L_\infty$-algebra gives Maurer-Cartan characterisation of Rota-Baxter $3$-Lie algebras.

It should be emphasized that   relative Rota-Baxter operators and (absolute) Rota-Baxter operators on $3$-Lie algebras are not compatible, so their relations remain to be elaborated \cite{HSZ22}.

\medskip

This paper is organised as follows. In Section 1, we recall some preliminaries about $3$-Lie algebras and their cohomology theory. In Section 2,
we consider  Rota-Baxter $3$-Lie algebras of arbitrary weight and introduce
their representations. We also provide various examples and new constructions.  In Section 3,  we  define a cohomology theory for  Rota-Baxter 3-Lie algebras. In Section 4,  we justify this cohomology theory by interpreting lower degree cohomology groups as formal deformations. In Section 5, we study  abelian extensions of
 Rota-Baxter 3-Lie algebras.   Section 6 is devoted to cohomological study of skeletal  Rota-Baxter $3$-Lie  2-algebras  and  crossed modules.

Throughout this paper, let $\bfk$ be a field of characteristic $0$.  Except specially stated,  vector spaces are  $\bfk$-vector spaces and  all    tensor products are taken over $\bfk$.

\bigskip

\section{Preliminaries}

We start with the background of 3-Lie algebras and their cohomology that we refer the reader to~\cite{Fil85, Kas87, Tak95, Rot05, AI10} for more details.

\begin{defn}(\cite{Fil85})
A {\bf 3-Lie algebra} is a vector space $\mathfrak{g}$  together with a skew-symmetric linear map
$[\cdot, \cdot, \cdot]_\mathfrak{g}: \wedge^3\mathfrak{g}\rightarrow \mathfrak{g}$ such that, for
 $x_i\in \mathfrak{g}, 1\leq i\leq 5$, the following {\bf Fundamental Identity} holds:
$$
  [x_1,x_2, [x_3,x_4, x_5]_\mathfrak{g}]_\mathfrak{g} = [[x_1,x_2, x_3]_\mathfrak{g},x_4, x_5]_\mathfrak{g}+[x_3,[x_1, x_2, x_4]_\mathfrak{g}, x_5]_\mathfrak{g}+[x_3,x_4, [x_1,x_2, x_5]_\mathfrak{g}]_\mathfrak{g}.
$$
\end{defn}

\smallskip
For $x_1, x_2\in \mathfrak{g}$, define $\ad_{x_1, x_2}\in \mathfrak{gl}(\mathfrak{g})$ by
\begin{align*}
\ad_{x_1, x_2} (x):=[x_1, x_2, x]_\mathfrak{g}.
\end{align*}
Then $\ad_{x_1, x_2}$ is a derivation for the bracket $[-,-,-]_\frakg$.

\begin{defn}(\cite{Kas87})
A {\bf representation} of a 3-Lie algebra $(\mathfrak{g}, [\cdot, \cdot, \cdot]_\mathfrak{g})$  on a vector space $M$ is a linear map: $\rho :\wedge^2 \mathfrak{g}\rightarrow \mathfrak{gl}(M)$, such that for all $x_1, x_2, x_3, x_4\in \mathfrak{g}$, the following equalities hold:
\begin{align*}
\rho(x_1, x_2)\circ\rho(x_3, x_4)&=\rho([x_1, x_2, x_3]_\mathfrak{g}, x_4)+\rho(x_3, [x_1, x_2, x_4]_\mathfrak{g})+\rho(x_3, x_4)\circ \rho(x_1, x_2),\\
\rho(x_1, [x_2, x_3, x_4]_\mathfrak{g})&=\rho(x_3, x_4)\circ \rho(x_1, x_2)-\rho(x_2, x_4)\circ \rho(x_1, x_3)+\rho(x_2, x_3)\circ \rho(x_1, x_4).
\end{align*}
\end{defn}

Let $(\mathfrak{g}, [\cdot, \cdot, \cdot]_\mathfrak{g})$ be a $3$-Lie algebra. The linear map $\ad:\mathfrak{g}\wedge \mathfrak{g}\rightarrow \mathfrak{gl}(\mathfrak{g})$ defines a representation of $(\mathfrak{g}, [\cdot, \cdot, \cdot]_\mathfrak{g})$ on itself, which is called the {\bf adjoint representation}.

Following \cite{AI10}, we recall the cohomology theory of $3$-Lie algebras.
Let $(M, \rho)$  be a representation of a $3$-Lie algebra $(\mathfrak{g}, [\cdot, \cdot, \cdot]_\mathfrak{g})$.  The  cochain complex of  $\mathfrak{g}$ with coefficients in $M$ is
 the  complex $\{ C^*_{\tLie}(\mathfrak{g}, M), \delta^* \}$, where $C^0_{\tLie} (\mathfrak{g}, M) =M$ and
\begin{eqnarray*}
C^n_{\tLie} (\mathfrak{g}, M) =\mbox{Hom}( \underbrace{\wedge^2\mathfrak{g}\otimes \cdots \otimes \wedge^2\mathfrak{g}}_{n-1}\wedge \mathfrak{g}, M ),  \forall n \geq 1,
\end{eqnarray*}
 and the coboundary operator $\delta^n:C^n_{\tLie}(\mathfrak{g}, M ) \rightarrow C^{n+1}_{\tLie}(\mathfrak{g}, M)$ is  given by
\begin{eqnarray*}
&&(\delta^n f)(\mathfrak{X}_1, \dots,  \mathfrak{X}_n, x_{n+1})\\
&=&\sum_{1\leq j< k\leq n}(-1)^{j}f(\mathfrak{X}_1, \dots, \mathfrak{\hat{X}}_j, \dots, \mathfrak{X}_{k-1}, [x_j, y_j, x_k]_\mathfrak{g}\wedge y_k
 + x_k\wedge [x_j, y_j, y_k]_\mathfrak{g}, \mathfrak{X}_{k+1}, \dots, \mathfrak{X}_{n}, x_{n+1})\\
&& +\sum^{n}_{j=1}(-1)^{j}f(\mathfrak{X}_1, \dots, \mathfrak{\hat{X}}_j, \dots, \mathfrak{X}_{n}, [x_j, y_j, x_{n+1}]_\mathfrak{g})\\
&&+ \sum^{n}_{j=1}(-1)^{j-1}\rho(x_j, y_j)f(\mathfrak{X}_1, \dots, \mathfrak{\hat{X}}_j, \dots, \mathfrak{X}_{n}, x_{n+1})\\
&&+ (-1)^{n+1}(\rho(y_n, x_{n+1})f(\mathfrak{X}_1,  \dots, \mathfrak{X}_{n-1}, x_{n})+\rho(x_{n+1}, x_n)f(\mathfrak{X}_1,  \dots, \mathfrak{X}_{n-1}, y_{n})),
\end{eqnarray*}
for $\mathfrak{X}_i=x_i\wedge y_i\in \wedge^2\mathfrak{g}, i=1, \dots,  n$ and $x_{n+1}\in \mathfrak{g}$, where $\mathfrak{\hat{X}}_j$ means that this element is deleted. The corresponding cohomology groups are called the  {\bf cohomology} of $\mathfrak{g}$ with coefficients in $M$, denoted by $\rmH^\bullet_{\mathrm{3}-\mathrm{Lie}}(\mathfrak{g}, M).$

 \bigskip

\section{Representations of  Rota-Baxter $3$-Lie algebras}

In  this section,    we consider  Rota-Baxter $3$-Lie algebras of arbitrary weight  and introduce their representations.
We also provide various examples and new constructions.

\begin{defn} (\cite{BGLW13})\label{Def: Rota-Baxter 3-Lie algebra} Let $\lambda\in \bfk$.
Given  a $3$-Lie algebra  $(\frakg,  [\cdot, \cdot, \cdot]_\frakg)$, it is called  a  \textbf{Rota-Baxter $3$-Lie algebra of  weight $\lambda$},  if there is a linear
operator $T: \frakg \rightarrow \frakg $ subjecting to
\begin{eqnarray}\label{Eq: Rota-Baxter relation}
[T(x),  T(y), T(z)]_\frakg&=&T\big([T(x), T(y), z]_\frakg+[T(x), y, T(z)]_\frakg+[x, T(y), T(z)]_\frakg\nonumber\\
&&+\lambda [T(x), y, z]_\frakg+\lambda [x, T(y), z]_\frakg+\lambda [x, y, T(z)]_\frakg\\
&&+\lambda^2 [x, y, z]_\frakg\big)\nonumber\end{eqnarray}	
for arbitrary $x, y, z \in \frakg $.

 A  Rota-Baxter $3$-Lie algebra of   weight $\lambda$ is a triple $(\frakg, [\cdot, \cdot, \cdot]_\frakg,  T)$ consisting of a $3$-Lie algebra $(\frakg,  [\cdot, \cdot, \cdot]_\frakg)$ together with a
 Rota-Baxter operator of  weight $\lambda$ on it.
\end{defn}
\begin{exam}
    For arbitrary $3$-Lie algebra $(\frakg,  [\cdot, \cdot, \cdot]_\frakg)$, the triple $(\frakg, [\cdot, \cdot, \cdot]_\frakg,  \id_\frakg)$ is a  Rota-Baxter $3$-Lie algebra of  weight $-1$.
\end{exam}

\begin{exam}
 Let $(\frakg, [\cdot, \cdot, \cdot]_\frakg,  T)$ be a  Rota-Baxter $3$-Lie algebra of  weight $\lambda$.
\begin{itemize}
\item[(a)]   For arbitrary $\lambda'\in \bfk$, the triple $(\frakg, [\cdot, \cdot, \cdot]_\frakg,  \lambda' T)$ is a Rota-Baxter $3$-Lie algebra of  weight $\lambda\cdot\lambda'$.

\item[(b)]    The triple  $(\frakg, [\cdot, \cdot, \cdot]_\frakg,  -\lambda\  \id_\frakg-T)$ is a  Rota-Baxter $3$-Lie algebra of  weight $\lambda$.

\item[(c)]   Given an automorphism $\psi\in$ Aut$(\frakg)$ of the $3$-Lie
algebra $\frakg$, the triple $(\frakg, [\cdot, \cdot, \cdot]_\frakg, \psi^{-1} \circ T \circ \psi)$ is a  Rota-Baxter $3$-Lie algebra of  weight $\lambda$.
\end{itemize}
\end{exam}

\begin{defn}
Let $(\frakg, [\cdot, \cdot, \cdot]_\frakg,  T)$ and $(\frakg', [\cdot, \cdot, \cdot]_{\frakg'},  T')$ be two  Rota-Baxter $3$-Lie algebras of  weight $\lambda$. A morphism from
$(\frakg, [\cdot, \cdot, \cdot]_\frakg,  T)$ to $(\frakg', [\cdot, \cdot, \cdot]_{\frakg'},  T')$ is a $3$-Lie algebra homomorphism $\phi: \frakg\rightarrow \frakg'$  satisfying  $\phi \circ T=T'\circ \phi$.

\end{defn}

Denote by $\RBA$ the category of  Rota-Baxter $3$-Lie algebras of  weight $\lambda$ with morphisms defined above.


\begin{defn}\label{Def: Rota-Baxter representations} Let $(\frakg, [\cdot, \cdot, \cdot]_\frakg, T)$ be a   Rota-Baxter $3$-Lie algebra of  weight $\lambda$ and $(M, \rho)$ be a representation over the  $3$-Lie
algebra $(\frakg, [\cdot, \cdot, \cdot]_\frakg)$. We say that $(M, \rho, T_M)$ is a \textbf{representation} over the  Rota-Baxter $3$-Lie algebra  $(\frakg, [\cdot, \cdot, \cdot]_\frakg, T)$ of  weight $\lambda$ if $(M, \rho)$ is endowed with a linear operator $T_M : M \rightarrow M$ such that the following equation
\begin{eqnarray*}\label{Def: Rota-Baxter representations}
\rho\big(T(x), T(y))(T_M(m)\big)& =& T_M\Big(\rho(T(x), T(y)) (m) +\rho(T(x), y)(T_M(m)) +\rho(x, T(y))(T_M(m)) \\
&&+  \lambda \rho(T(x), y) (m)+\lambda \rho(x, T(y)) (m)+\lambda \rho(x, y) (T_M(m))\\
&&+\lambda^2\rho(x, y)(m)\Big)
\end{eqnarray*}
holds for any  $x, y \in \frakg$ and $m \in M$.

\end{defn}
\begin{exam}
It is obvious  that  $(\frakg ,[\cdot, \cdot, \cdot]_\frakg, T)$ itself is a representation over the  Rota-Baxter $3$-Lie algebra  $(\frakg ,[\cdot, \cdot, \cdot]_\frakg, T)$ of  weight $\lambda$, called the regular representation or the adjoint representation.
\end{exam}
\begin{exam}
Let $(\frakg ,[\cdot, \cdot, \cdot]_\frakg)$ be a $3$-Lie algebra and $(M, \rho)$  be a representation over it. Then the triple $(M, \rho, \id_M)$ is a
representation of the  Rota-Baxter $3$-Lie algebra  $(\frakg ,[\cdot, \cdot, \cdot]_\frakg, \id_\frakg)$ of  weight $-1$.
\end{exam}

\begin{exam}
Let $(\frakg ,[\cdot, \cdot, \cdot]_\frakg, T)$ be a  Rota-Baxter  $3$-Lie algebra of weight $\lambda$ and $(M, \rho,T_m)$  be a representation over it. Then for arbitrary scalar $\mu\in \bfk$, the triple $(M, \rho, \mu T_M)$ is a representation of the  Rota-Baxter $3$-Lie algebra  $(\frakg ,[\cdot, \cdot, \cdot]_\frakg, \mu T)$ of  weight $(\mu \lambda)$.
\end{exam}


\begin{prop}
Let $(\frakg ,[\cdot, \cdot, \cdot]_\frakg, T)$ be a   Rota-Baxter $3$-Lie algebra of  weight $\lambda$ and $\{(M_i, \rho_i, T_{M_i})\}_{i\in I}$  be a  family of representation of it. Then  the triple $(\oplus_{i\in I}M_i,  (\rho_i)_{i\in I}, \oplus_{i\in I}T_{M_i})$ is a representation of the  Rota-Baxter $3$-Lie algebra  $(\frakg ,[\cdot, \cdot, \cdot]_\frakg, T)$ of weight $\lambda$.
\end{prop}
\begin{proof}
For arbitrary  $x, y\in \frakg$ and $m=(m_i)_{i\in I}\in \oplus_{i\in I}M_i$, we have
\begin{eqnarray*}
&& (\rho_i)_{i\in I}(T(x), T(y))(\oplus_{i\in I}T_{M_i})(m_i)_{i\in I}\\
&=& \left( \rho_i(T(x), T(y))(T_{M_i})(m_i)\right) _{i\in I}\\
&=& \big(T_{M_i}(\rho_i(T(x), T(y)) (m) +\rho_i(T(x), y)(T_M(m)) +\rho_i(x, T(y))(T_M(m)) \nonumber\\
&&+  \lambda \ \rho_i(T(x), y) (m)+\lambda \ \rho_i(x, T(y))( m)+\lambda \ \rho_i(x, y) (T_M(m))+\lambda^2\rho_i(x, y)(m))\big)_{i\in I}\\
&=& (\oplus_{i\in I}T_{M_i})\big((\rho_i)_{i\in I}(T(x), T(y))( m) +(\rho_i)_{i\in I}(T(x), y)(T_M(m)) +(\rho_i)_{i\in I}(x, T(y))(T_M(m)) \nonumber\\
&&+  \lambda \ (\rho_i)_{i\in I}(T(x), y)( m)+\lambda \ (\rho_i)_{i\in I}(x, T(y)) (m)+\lambda \ (\rho_i)_{i\in I}(x, y) (T_M(m))\\
&&+\lambda^2(\rho_i)_{i\in I}(x, y)(m)\big).
\end{eqnarray*}
Hence, $(\oplus_{i\in I}M_i,  (\rho_i)_{i\in I}, \oplus_{i\in I}T_{M_i})$ is a representation of the  Rota-Baxter $3$-Lie algebra  $(\frakg ,[\cdot, \cdot, \cdot]_\frakg, T)$ of  weight $\lambda$.
\end{proof}

In the following, we construct the semidirect product in the context of  Rota-Baxter $3$-Lie
algebras of  weight $\lambda$.

\begin{prop}
Let $(\frakg ,[\cdot, \cdot, \cdot]_\frakg, T)$ be a   Rota-Baxter $3$-Lie algebra of  weight $\lambda$ and $(M, \rho, T_M)$  be a representation of it. Then $(\frakg\oplus M, T\oplus T_M )$ is a Rota-Baxter $3$-Lie algebra of  weight $\lambda$, where the Lie bracket on $\frakg\oplus M$ is given
by the semidirect product
\begin{eqnarray*}
 [x+u, y+v, z+w]_{\frakg\oplus M}&=&[x, y, z]_\frakg +\rho(x, y)w+\rho(y, z)u+\rho(z, x)v,
\end{eqnarray*}
for arbitrary   $x, y, z \in \frakg$ and $u, v, w \in M$.
\end{prop}
\begin{proof}
For arbitrary   $x, y, z \in \frakg$ and $u, v, w \in M$, we have
\begin{eqnarray*}
&& [(T\oplus T_M)(x+u), (T\oplus T_M)(y+v), (T\oplus T_M)(z+w)]_{\frakg\oplus M}\\
&=& [T(x), T(y), T(z)]_\frakg +\rho(T(x), T(y))T_M(w)+\rho(T(y), T(z))T_M(u)+\rho(T(z), T(x))T_M(v)\\
&=&T\big([T(x), T(y), z]_\frakg+[T(x), y, T(z)]_\frakg+[x, T(y), T(z)]_\frakg+\lambda \ [T(x), y, z]_\frakg\\
&&+\lambda \ [x, T(y), z]_\frakg+\lambda \ [x, y, T(z)]_\frakg+\lambda^2 [x, y, z]_\frakg\big)\\
&&+T_M\big(\rho(T(x), T(y)) w+\rho(T(x), y)T_M(w) +\rho(x, T(y))(T_M(w)) +  \lambda \ \rho(T(x), y)( w)\\
&&+\lambda \ \rho(x, T(y))(w) +\lambda \ \rho(x, y) (T_M(w))+\lambda^2\rho(x, y)(w)\big)\\
&&+ T_M\big(\rho(T(y), T(z))(u) +\rho(T(y), z)(T_M(u)) +\rho(y, T(z))(T_M(u))+  \lambda \ \rho(T(y), z) (u)\\
&&+\lambda \ \rho(y, T(z))(u)+\lambda \ \rho(y, z) (T_M(u))+\lambda^2\rho(y, z)(u)\big)\\
&&+ T_M\big(\rho(T(z), T(x)) (v) +\rho(T(z), x)T_M(v) +\rho(z, T(x))(T_M(v))\\
&&+  \lambda \ \rho(T(z), x)(v)+\lambda \ \rho(z, T(x))( v)+\lambda \ \rho(z, x)( T_M(v))+\lambda^2\rho(z, x)(v)\big)\\
&=& (T\oplus T_M)\big([(T\oplus T_M)(x+u), (T\oplus T_M)(y+v), z+w]_{\frakg\oplus M}\\
&&+[(T\oplus T_M)(x+u), y+v, (T\oplus T_M)(z+w)]_{\frakg\oplus M}+[x+u, (T\oplus T_M)(y+v), (T\oplus T_M)(z+w)]_{\frakg\oplus M}\\
&&+\lambda \ [(T\oplus T_M)(x+u), y+v, z+w]_{\frakg\oplus M}+\lambda \ [x+u, (T\oplus T_M)(y+v), z+w]_{\frakg\oplus M}\\
&&+\lambda \ [x+u, y+v, (T\oplus T_M)(z+w)]_{\frakg\oplus M}+\lambda^2 [x+u, y+v, z+w]_{\frakg\oplus M}\big).
\end{eqnarray*}
This shows that  $T\oplus T_M$ is a  Rota-Baxter operator  of  weight $\lambda$ on the semidirect product $3$-Lie algebra. Hence
the result follows.
\end{proof}

\begin{prop} (\cite{BGLW13})\label{Prop: new RB 3-Lie algebra}
	Let $(\frakg, [\cdot, \cdot, \cdot]_\frakg, T)$ be a  Rota-Baxter $3$-Lie algebra of  weight $\lambda$. Define a new trinary operation as:
\begin{eqnarray*}
[x, y, z]_T:&=&[T(x), T(y), z]_\frakg+[T(x), y, T(z)]_\frakg+[x, T(y), T(z)]_\frakg\\
&&+\lambda \ [T(x), y, z]_\frakg+\lambda \ [x, T(y), z]_\frakg+\lambda \ [x, y, T(z)]_\frakg+\lambda^2 [x, y, z]_\frakg,
\end{eqnarray*}
for arbitrary  $x, y, z\in \frakg $. Then
\begin{itemize}

\item[(a)]      $(\frakg ,[\cdot, \cdot, \cdot]_T )$ is a new  $3$-Lie algebra, we denote this $3$-Lie algebra by $\frakg_T$;

    \item[(b)] the triple  $(\frakg_T ,[\cdot, \cdot, \cdot]_T, T)$ is a  Rota-Baxter $3$-Lie algebra of  weight $\lambda$;

\item[(c)] the map $T:(\frakg_T ,[\cdot, \cdot, \cdot]_T, T)\rightarrow (\frakg,[\cdot, \cdot, \cdot]_\frakg, T)$ is a  morphism of  Rota-Baxter $3$-Lie algebras of  weight $\lambda$.
    \end{itemize}
	\end{prop}
\begin{proof}
(a)  On can show it directly by a tedious computation;

(b) We observe that
\begin{eqnarray*}
&&[T(x), T(y), T(z)]_T\\
&=&[T^2(x), T^2(y), T(z)]_\frakg+[T^2(x), T(y), T^2(z)]_\frakg+[T(x), T^2(y), T^2(z)]_\frakg\\
&&+\lambda \ [T^2(x), T(y), T(z)]_\frakg+\lambda \ [T(x), T^2(y), T(z)]_\frakg+\lambda \ [T(x), T(y), T^2(z)]_\frakg+\lambda^2 [T(x), T(y), T(z)]_\frakg\\
&=& T\Big([T(x), T(y), z]_T+[T(x), y, T(z)]_T+[x, T(y), T(z)]_T\\
&&+\lambda \ [T(x), y, z]_T+\lambda \ [x, T(y), z]_T+\lambda \ [x, y, T(z)]_T+\lambda^2 [x, y, z]_T\Big),
\end{eqnarray*}
which shows that $T$ is a  Rota-Baxter operator of   weight $\lambda$ on the $3$-Lie algebra $\frakg_T$.

(c) Since $T$ is a  Rota-Baxter operator of   weight $\lambda$ on $\frakg$, it follows from Equation~\eqref{Eq: Rota-Baxter relation} that
\begin{eqnarray*}
T([x, y, z]_T)=[T(x), T(y), T(z)]_\frakg, ~~\mbox{for}~~ x, y, z\in \frakg.
\end{eqnarray*}
This implies that $T:(\frakg_T ,[\cdot, \cdot, \cdot]_T, T)\rightarrow (\frakg,[\cdot, \cdot, \cdot]_\frakg, T)$ is a  morphism of  Rota-Baxter $3$-Lie algebra of  weight $\lambda$.
\end{proof}

In the following, we will introduce new Rota-Baxter representations that will be useful in the next section to construct
the cohomology of  Rota-Baxter $3$-Lie algebras of  arbitrary weights.

\begin{thm}\label{Prop: new-representation}
Let $(\frakg ,[\cdot, \cdot, \cdot]_\frakg, T)$ be a  Rota-Baxter $3$-Lie algebra of weight $\lambda$ and $(M, \rho, T_M)$  be a representation of it. Define a map $\rho_T:\wedge^2 \frakg \rightarrow$ End$(M)$ by
\begin{eqnarray*}
\rho_T(x,y)m:=\rho\big(T(x), T(y)\big) m-T_M\big(\rho(T(x), y) m +\rho(x,T( y))m+\lambda\rho(x,y)m\big),
\end{eqnarray*}
for arbitrary  $x, y\in \frakg, m\in M$. Then $\rho_T$  defines a representation of the $3$-Lie algebra $(\frakg_T ,[\cdot, \cdot, \cdot]_T)$ on $M$. Moreover, $(M, \rho_T, T_M)$ is a representation
of the  Rota-Baxter $3$-Lie algebra  $(\frakg_T ,[\cdot, \cdot, \cdot]_T, T)$ of   weight $\lambda$.

\end{thm}
\begin{proof} One can show that  $\rho_T$  defines a representation of the $3$-Lie algebra $(\frakg_T ,[\cdot, \cdot, \cdot]_T)$ on $M$ directly by a tedious computation.   Moreover, we have
\begin{small}
\begin{eqnarray*}
&& \rho_T(T(x),T(y))T_{M}(m)\\
&=& \rho(T^2(x), T^2(y))T_{M}(m)-T_M(\rho(T^2(x),T(y)) T_{M}(m) +\rho(T(x),T^2(y))T_{M}(m)+\lambda\rho(T(x),T(y))T_{M}(m))\\
&=& T_M\big(\rho(T^2(x), T^2(y)) m +\rho(T^2(x), T(y))T_M(m) +\rho(T(x), T^2(y))T_M(m) \\
&&+  \lambda \ \rho(T^2(x), T(y)) m+\lambda \ \rho(T(x), T^2(y)) m+\lambda \ \rho(T(x), T(y)) T_M(m)+\lambda^2\rho(T(x), T(y))m\big)\\
&&-T^2_M\big(\rho(T^2(x), T(y)) m +\rho(T^2(x), y)T_M(m) +\rho(T(x), T(y))T_M(m)+  \lambda \ \rho(T^2(x), y) m\\
&&+\lambda \ \rho(T(x), T(y)) m+\lambda \ \rho(T(x), y) T_M(m)+\lambda^2\rho(T(x), y)m+\rho(T(x), T^2(y)) m +\rho(T(x), T(y))T_M(m) \\
&&+\rho(x, T^2(y))T_M(m)+  \lambda \ \rho(T(x), T(y)) m+\lambda \ \rho(x, T^2(y)) m+\lambda \ \rho(x, T(y)) T_M(m)+\lambda^2\rho(x, T(y))m\\
&& +\lambda\rho(T(x), T(y)) m +\lambda\rho(T(x), y)T_M(m) +\lambda\rho(x, T(y))T_M(m) +  \lambda^2 \rho(T(x), y) m+\lambda^2 \rho(x, T(y)) m\\
&&+\lambda^2 \rho(x, y) T_M(m)+\lambda^3\rho(x, y)m\big)\\
&=& T_M\big(\rho_T(T(x), T(y)) m +\rho_T(T(x), y)T_M(m) +\rho_T(x, T(y))T_M(m)\\
&&+  \lambda \ \rho_T(T(x), y) m+\lambda \ \rho_T(x, T(y)) m+\lambda \ \rho_T(x, y) T_M(m)+\lambda^2\rho_T(x, y)m\big),
\end{eqnarray*}
\end{small}
which shows $(M, \rho_T, T_M)$ is a representation
of the  Rota-Baxter $3$-Lie algebra  $(\frakg_T, [\cdot, \cdot, \cdot]_T, T)$ of  weight $\lambda$.
\end{proof}

\bigskip

\section{Cohomology theory of  Rota-Baxter 3-Lie algebras} \label{Sect: Cohomology theory of Rota-Baxter 3-Lie algebras}
In this  section, we will define a cohomology theory for   Rota-Baxter $3$-Lie algebras of  arbitrary weights.

\subsection{Cohomology of  Rota-Baxter operators}\
\label{Subsect: cohomology RB operator}

Firstly, let's introduce the cohomology of  Rota-Baxter
operators of  arbitrary weights.

 Let $(\frakg , [\cdot, \cdot, \cdot]_\frakg , T)$ be a  Rota-Baxter $3$-Lie algebra of  weight $\lambda$ and $(M, \rho, T_M)$ be a representation over it. Recall that
Proposition~\ref{Prop: new RB 3-Lie algebra}  and Proposition~\ref{Prop: new-representation}  give a new
$3$-Lie algebra   $\frakg _T $ and
  a new  representation  $M$ over $\frakg _T $.
 Consider the  cochain complex of $\frakg _T $ with
 coefficients in $M$:
 $$\C^\bullet_{\PLA}(\frakg _T, M)=\bigoplus\limits_{n=0}^\infty \C^n_{\PLA}(\frakg_T , M).$$
  More precisely,  for $n\geqslant 0$,  $ \C^n_{\PLA}(\frakg _T , M)=\Hom  (\underbrace{\wedge^2\mathfrak{g}\otimes \cdots \otimes \wedge^2\mathfrak{g}}_{n-1}\wedge \mathfrak{g}, M)$ and its differential $$\partial^n:
 \C^n_{\tLie}(\frakg _T,\  M)\rightarrow  \C^{n+1}_{\tLie}(\frakg _T , M) $$ is defined as:
\begin{small}
\begin{eqnarray*}
&&(\partial^n f)(\mathfrak{X}_1, ... \mathfrak{X}_n, x_{n+1})\\
&=&\sum_{1\leq j< k\leq n}(-1)^{j}f\big(\mathfrak{X}_1, \cdots, \mathfrak{\hat{X}}_j, \cdots, \mathfrak{X}_{k-1}, [x_j, y_j, x_k]_T\wedge y_k+ x_k\wedge [x_j, y_j, y_k]_T, \mathfrak{X}_{k+1}, \cdots, \mathfrak{X}_{n}, x_{n+1}\big)\\
&& +\sum^{n}_{j=1}(-1)^{j}f\big(\mathfrak{X}_1, \cdots, \mathfrak{\hat{X}}_j, \cdots, \mathfrak{X}_{n}, [x_j, y_j, x_{n+1}]_T\big)+ \sum^{n}_{j=1}(-1)^{j+1}\rho_T(x_j, y_j)f\big(\mathfrak{X}_1, \cdots, \mathfrak{\hat{X}}_j, \cdots, \mathfrak{X}_{n}, x_{n+1}\big)\\
&&+ (-1)^{n+1}\Big(\rho_T(y_n, x_{n+1})f\big(\mathfrak{X}_1,  \cdots, \mathfrak{X}_{n-1}, x_{n}\big)+\rho_T(x_{n+1}, x_n)f\big(\mathfrak{X}_1,  \cdots, \mathfrak{X}_{n-1}, y_{n}\big)\Big),\\
&=&\sum_{1\leq j< k\leq n}(-1)^{j}f(\mathfrak{X}_1, \cdots, \mathfrak{\hat{X}}_j, \cdots, \mathfrak{X}_{k-1}, [x_j, y_j, x_k]_T\wedge y_k+ x_k\wedge [x_j, y_j, y_k]_T, \mathfrak{X}_{k+1}, \cdots, \mathfrak{X}_{n}, x_{n+1})\\
&& +\sum^{n}_{j=1}(-1)^{j}f(\mathfrak{X}_1, \cdots, \mathfrak{\hat{X}}_j, \cdots, \mathfrak{X}_{n}, [x_j, y_j, x_{n+1}]_T)+ \sum^{n}_{j=1}(-1)^{j+1}\rho (T(x_j), T(y_j))f(\mathfrak{X}_1, \cdots, \mathfrak{\hat{X}}_j, \cdots, \mathfrak{X}_{n}, x_{n+1})\\
&&-\sum^{n}_{j=1}(-1)^{j+1} T_M\left( \rho(T(x_j),  y_j )f(\mathfrak{X}_1, \cdots, \mathfrak{\hat{X}}_j, \cdots, \mathfrak{X}_{n}, x_{n+1})\right) \\
&&-\sum^{n}_{j=1}(-1)^{j+1} T_M\left( \rho( x_j ,  T(y_j) ) f(\mathfrak{X}_1, \cdots, \mathfrak{\hat{X}}_j, \cdots, \mathfrak{X}_{n}, x_{n+1})\right) \\
&&-\sum^{n}_{j=1}(-1)^{j+1}\lambda \ T_M\left( \rho( x_j ,  y_j ) f(\mathfrak{X}_1, \cdots, \mathfrak{\hat{X}}_j, \cdots, \mathfrak{X}_{n}, x_{n+1})\right) \\
&&+ (-1)^{n+1}\left(  \rho(T(y_n), T(x_{n+1}))f(\mathfrak{X}_1,  \cdots, \mathfrak{X}_{n-1}, x_{n})+\rho (T(x_{n+1}), T(x_n))f(\mathfrak{X}_1,  \cdots, \mathfrak{X}_{n-1}, y_{n}) \right)  \\
&&- (-1)^{n+1} T_M\left( \rho(T(y_n),  x_{n+1} )f(\mathfrak{X}_1,  \cdots, \mathfrak{X}_{n-1}, x_{n})+\rho (T(x_{n+1}), x_n)f(\mathfrak{X}_1,  \cdots, \mathfrak{X}_{n-1}, y_{n})\right) \\
&&- (-1)^{n+1}T_M\left(  \rho( y_n , T(x_{n+1}))f(\mathfrak{X}_1,  \cdots, \mathfrak{X}_{n-1}, x_{n})+\rho (x_{n+1}, T(x_n))f(\mathfrak{X}_1,  \cdots, \mathfrak{X}_{n-1}, y_{n})\right) \\
&&- (-1)^{n+1}\lambda \ T_M\left(  \rho( y_n ,  x_{n+1} )f(\mathfrak{X}_1,  \cdots, \mathfrak{X}_{n-1}, x_{n})+\rho (x_{n+1}, x_n)f(\mathfrak{X}_1,  \cdots, \mathfrak{X}_{n-1}, y_{n})\right)
\end{eqnarray*}
\end{small}
for $\mathfrak{X}_i=x_i\wedge y_i\in \wedge^2\mathfrak{g}, i=1, 2\cdots n$ and $x_{n+1}\in \mathfrak{g}$.
 \smallskip

 \begin{defn}
 	Let $(\frakg , [\cdot, \cdot, \cdot]_\frakg , T)$ be a  Rota-Baxter $3$-Lie algebra of   weight $\lambda$ and $(M, \rho, T_M)$ be a representation over it.  Then the cochain complex $(\C^\bullet_\PLA(\frakg _T,  M),\partial)$ is called the \textbf{cochain complex of  Rota-Baxter operator}  $T$  of  weight $\lambda$ with coefficients in $(M, \rho_T, T_M)$,  denoted by $C_{\RBO}^\bullet(\frakg , M)$. The cohomology of $C_{\RBO}^\bullet(\frakg ,M)$, denoted by $\mathrm{H}_{\RBO}^\bullet(\frakg ,M)$, are called the \textbf{cohomology of    Rota-Baxter operator} $T$ of  weight $\lambda$  with coefficients in $(M, \rho_T, T_M)$.
 	
 	 When $(M,\rho, T_M)$ is the regular representation  $ (\frakg , [\cdot, \cdot, \cdot]_\frakg,   T)$, we denote $\C^\bullet_{\RBO}(\frakg ,\frakg )$ by $\C^\bullet_{\RBO}(\frakg )$ and call it the cochain complex of Rota-Baxter operator of  weight $\lambda$, and denote $\rmH^\bullet_{\RBO}(\frakg ,\frakg )$ by $\rmH^\bullet_{\RBO}(\frakg )$ and call it the cohomology of  Rota-Baxter operator  $T$ of weight $\lambda$.
 \end{defn}

\subsection{Cohomology of  Rota-Baxter $3$-Lie algebras}\
\label{Subsec:chomology RB}

In this subsection, we will combine the  cohomology of   $3$-Lie algebras and the cohomology of  Rota-Baxter operators  of  arbitrary weights to define a cohomology theory for   Rota-Baxter $3$-Lie algebras  of  arbitrary weights.

Let $M=(M,\rho, T_M)$ be a representation over a   Rota-Baxter $3$-Lie algebra     $\frakg =(\frakg ,\mu=[\cdot, \cdot, \cdot]_\frakg,T)$ of weight $\lambda$. Now, let's construct a chain map   $$\Phi^\bullet:\C^\bullet_{\PLA}(\frakg ,M) \rightarrow C_{\RBO}^\bullet(\frakg ,M),$$ i.e., the following commutative diagram:
\[\xymatrix{
		\C^0_{\PLA}(\frakg ,M)\ar[r]^-{\delta^0}\ar[d]^-{\Phi^0}& \C^1_{\PLA}(\frakg,M)\ar@{.}[r]\ar[d]^-{\Phi^1}&\C^n_{\PLA}(\frakg,M)\ar[r]^-{\delta^n}\ar[d]^-{\Phi^n}&\C^{n+1}_{\PLA}(\frakg,M)\ar[d]^{\Phi^{n+1}}\ar@{.}[r]&\\
		\C^0_{\RBO}(\frakg,M)\ar[r]^-{\partial^0}&\C^1_{\RBO}(\frakg,M)\ar@{.}[r]& \C^n_{\RBO}(\frakg,M)\ar[r]^-{\partial^n}&\C^{n+1}_{\RBO}(\frakg,M)\ar@{.}[r]&
.}\]

Define $\Phi^0=\Id_{\Hom(k,M)}=\Id_M$, $\Phi^1=f\circ T-T_M\circ f$ and for  $n\geqslant 2$ and $ f\in \C^n_{\PLA}(\frakg,M)$,  define $\Phi^n(f)\in \C^n_{\RBO}(\frakg,M)$ as:
\begin{align*}
  \Phi^n(f) =&f\circ\left( T,\cdots,T,T\right)\circ \left( (\Id \wedge \Id)^{\otimes n-2 }  \otimes \Id^{\wedge 3 } \right) \\
 &-\sum_{k=0}^{2n-2}\lambda^{2n-k-2}\sum_{1\leqslant i_1<i_2<\cdots< i_{k-1}\leqslant 2n-1}  \\
 & T_M\circ f\circ (\Id^{(i_1-1)} ,  T , \Id^{ (i_2-i_1-1)} , T, \cdots, T,\Id^{  (2n-1-i_{k})}  )\circ \left( (\Id \wedge \Id)^{\otimes n-2 }  \otimes \Id^{\wedge 3 } \right)  ,
\end{align*}
where the operators  $ \Id^{(i_1-1)} ,  T , \Id^{ (i_2-i_1-1)} , T, \cdots, T,\Id^{  (2n-1-i_{k})}  $ successively operate on the elements $x_1,y_1,\dots,x_n,y_n,x_{n+1}$.

\smallskip

\begin{prop}\label{Prop: Chain map Phi}
	The map $\Phi^\bullet: \C^\bullet_\PLA(\frakg,M)\rightarrow \C^\bullet_{\RBO}(\frakg,M)$ is a chain map.
\end{prop}


 We leave the long proof of this result to Appendix A.

\smallskip

\begin{defn}
 Let $M=(M,\rho, T_M)$ be a  representation over a   Rota-Baxter $3$-Lie algebra   $\frakg =(\frakg ,\mu=[\cdot, \cdot, \cdot]_\frakg,T)$ of   weight $\lambda$.  We define the  cochain complex $(\C^\bullet_{\RBA}(\frakg,M), d^\bullet)$  of  Rota-Baxter $3$-Lie algebra  $(\frakg,\mu,T)$ of  weight $\lambda$ with coefficients in $(M,\rho, T_M)$ to the negative shift of the mapping cone of $\Phi^\bullet$, that is,   let
\[\C^0_{\RBA}(\frakg,M)=\C^0_\PLA(\frakg,M)  \quad  \mathrm{and}\quad   \C^n_{\RBA}(\frakg,M)=\C^n_\PLA(\frakg,M)\oplus \C^{n-1}_{\RBO}(\frakg,M), \forall n\geqslant 1,\]
 and the differential $d^n: \C^n_{\RBA}(\frakg,M)\rightarrow \C^{n+1}_{\RBA}(\frakg,M)$ is given by \[d^n(f,g)= (\delta^n(f), -\partial^{n-1}(g)  -\Phi^n(f))\]
 for arbitrary  $f\in \C^n_\PLA(\frakg,M)$ and $g\in \C^{n-1}_{\RBO}(\frakg,M)$.
The  cohomology of $(\C^\bullet_{\RBA}(\frakg,M), d^\bullet)$, denoted by $\rmH_{\RBA}^\bullet(\frakg,M)$,  is called the \textbf{cohomology of the  Rota-Baxter $3$-Lie algebra}  $(\frakg,\mu,T)$ of  weight $\lambda$ with coefficients in $(M,\rho, T_M)$.
When $(M,\rho, T_M)=(\frakg,\mu, T)$, we just denote $\C^\bullet_{\RBA}(\frakg,\frakg), $\\$\rmH^\bullet_{\RBA}(\frakg,\frakg)$   by $\C^\bullet_{\RBA}(\frakg),~ \rmH_{\RBA}^\bullet(\frakg)$ respectively, and call  them the \textbf{cochain complex, the cohomology of   Rota-Baxter $3$-Lie algebra}  $(\frakg,\mu,T)$ of  weight $\lambda$ respectively.
\end{defn}
There is an obvious short exact sequence of complexes:
$$ 0\to \C^\bullet_{\RBO}(\frakg,M)[-1]\to \C^\bullet_{\RBA}(\frakg,M)\to \C^\bullet_{\PLA}(\frakg,M)\to 0,$$
which induces a long exact sequence of cohomology groups
$$0\to \rmH^{0}_{\RBA}(\frakg, M)\to \rmH^0_{\mathrm{3}-\mathrm{Lie}}(\frakg, M)\to \rmH^0_{\RBO}(\frakg, M) \to \rmH^{1}_{\RBA}(\frakg, M)\to {\rmH}^1_{\mathrm{3}-\mathrm{Lie}}(\frakg, M)\to\cdots  $$
$$\cdots\to {\rmH}^p_{\mathrm{3}-\mathrm{Lie}}(\frakg, M)\to \rmH^p_{\RBO}(\frakg, M)\to \rmH^{p+1}_{\RBA}(\frakg, M)\to {\rmH}^{p+1}_{\mathrm{3}-\mathrm{Lie}}(\frakg, M)\to \cdots$$


\bigskip


\section{Formal deformations of  Rota-Baxter $3$-Lie algebras}

In this section, we will study formal deformations of  Rota-Baxter $3$-Lie algebras of  arbitrary weights and interpret  them  via    lower degree   cohomology groups  of  Rota-Baxter $3$-Lie algebras   defined in last section.

\subsection{Formal deformations of Rota-Baxter operator  with $3$-Lie bracket fixed}\

Let $(\frakg,\mu, T)$ be a Rota-Baxter $3$-Lie algebra of weight $\lambda$. Let us consider the case where  we only deform the Rota-Baxter operator with the $3$-Lie bracket fixed. In this case,     $\frakg[[t]]=\{\sum_{i=0}^\infty a_it^i\ | \ a_i\in \frakg, \forall i\geqslant 0\}$ is endowed with the $3$-Lie bracket induced from that of $\frakg$, say,
$$\mu(\sum_{i=0}^\infty a_it^i,\sum_{j=0}^\infty b_jt^j,\sum_{i=0}^\infty c_it_i)=\sum_{n=0}^\infty \big(\sum_{i+j+k=n\atop i,j,k\geqslant 0} \mu (a_i,b_j,c_k)\big)t^n.$$
Then $\frakg[[t]]$ becomes a $3$-Lie algebra over $\bfk[[t]]$, whose $3$-Lie bracket is still denoted by $\mu$.

 Consider a 1-parameterized family:
\[  T_t=\sum_{i=0}^\infty T_it^i,  \ T_i\in \C^1_{\RBO}(\frakg).\]
\begin{defn}
	A  \textbf{$1$-parameter formal deformation of the   Rota-Baxter operator} $T$ on a $3$-Lie algebra  $(\frakg, \mu)$ is a  family $T_t$ which is a $\bfk[[t]]$-linear Rota-Baxter operator on the $3$-Lie algebra   $\frakg[[t]]$  such that $T_0=T$. The operator  $T_1$ is called the \textbf{infinitesimal} of the 1-parameter formal deformation $(\frakg[[t]],T_t)$ of   Rota-Baxter operators on the $3$-Lie algebra $(\frakg,\mu)$.
\end{defn}

Power series   $ T_t$ determine a  $1$-parameter formal deformation of  the   Rota-Baxter operator $T$ on a $3$-Lie algebra  $(\frakg, \mu)$ if and only if  the following equation holds:
 $$\begin{array}{cl}
 & \mu(T_t(x_1),  T_t(x_2), T_t(x_3))\\
 =&T_t(\mu(T_t(x_1), T_t(x_2), x_3)+\mu(T_t(x_1), x_2, T_t(x_3))+\mu(x_1, T_t(x_2), T_t(x_3))\\
&+\lambda \  \mu(T_t(x_1), x_2, x_3)+\lambda \  \mu(x_1, T_t(x_2), x_3)+\lambda \  \mu(x_1, x_2, T_t(x_3))+\lambda^2  \mu(x_1, x_2, x_3)\big),
 \end{array}$$
  for arbitrary  $x_1, x_2, x_3, x_4, x_5\in \frakg$.
 Expanding  this equation  and comparing the coefficient of $t^n$, we obtain  that  $\{T_i\}_{i\geqslant0}$ have to  satisfy: for arbitrary $n\geqslant 0$,
 \begin{small}
 \begin{equation}
 	\begin{aligned}	\label{Eq: deform eq for RBO}
 		&\sum_{i+j+k=n\atop i, j,k\geqslant 0} 	\mu (T_i(x_1),  T_{j}(x_2),  T_{k}(x_3))\\
 =&\sum_{i+j+k=n\atop i, j,k\geqslant 0} \left( T_{ i}  (\mu  (x_1,  T_{j}(x_2),  T_{k}(x_3))) +T_{i}(\mu  (T_{j}(x_1),  x_2,  T_{k}(x_3)))+T_{k} (\mu  ( T_{i}(x_1),  T_{j}(x_2),  x_3))\right) \\
 		&+\lambda \sum_{i+j=n\atop i, j\geqslant 0}\left( T_i( \mu (T_{j}(x_1),  x_2,  x_3))+ T_i( \mu ( x_1,  T_{j}(x_2),  x_3))+ T_i( \mu ( x_1,  x_2,  T_{j}(x_3)))\right) \\
 		&+\lambda^2T_n( \mu(x_1,  x_2,  x_3)).
 	\end{aligned}
 \end{equation}
 \end{small}

Obviously, when $n=0$,  Equation~(\ref{Eq: deform eq for RBO}) becomes  exactly Equation~(\ref{Eq: Rota-Baxter relation}) defining Rota-Baxter operator $T=T_0$.

When $n=1$,    Equation~(\ref{Eq: deform eq for RBO}) has the form
   \begin{small}
    \begin{equation}
 	\begin{aligned}	\label{Eq: deform eq1 for RBO}
 		&\mu \left( T_1(x_1),  T(x_2),  T(x_3)\right)+\mu \left(T(x_1),  T_1(x_2),  T(x_3)\right)+\mu \left(T(x_1),  T(x_2),  T_1(x_3)\right) \\
 		=& \quad  T\left(\mu(x_1,  T_1(x_2),  T(x_3))+\mu(x_1,  T(x_2),  T_1(x_3))\right)\\
      & +T\left(\mu(T_{1}(x_1),  x_2,  T(x_3))
 		  +\mu(T(x_1),  x_2,  T_{1}(x_3))\right)\\
 &+T\left(\mu(T_{1}(x_1),  T(x_2),  x_3)+\mu(T(x_1),  T_{1}(x_2),  x_3)\right)   \\
 	&+T_1 \left(  \mu  (x_1,  T(x_2),  T(x_3))+ \mu  (T(x_1),  x_2,  T(x_3))+ \mu  (T(x_1),  T(x_2),  x_3)\right) \\
 &+\lambda T \left(  \mu  (T_1(x_1),  x_2,  x_3)+\mu(x_1, T( x_2) ,  x_3)+  \mu( x_1 ,  x_2 ,  T_1(x_3) )\right)  \\
 	&+\lambda T_1  \left(  \mu  (x_1,  T(x_2) ,  x_3)+\mu  (T(x_1) ,  x_2,  x_3)  +\mu  (x_1 ,  x_2 ,  T(x_3)  )\right) \\
 &+ \lambda^2 T_1( \mu( x_1,  x_2 ,  x_3))
 \end{aligned}
 \end{equation}
\end{small}
 which     says exactly that $\partial^1(T_1)=0\in \C^\bullet_{\RBO}(\frakg)$.
This proves the following result:
\begin{prop}
	Let $ T_t $ be a  1-parameter formal deformation of
	Rota-Baxter operator  $ T $ of weight $\lambda$. Then
	$ T_1 $ is a 1-cocycle in the cochain complex
	$\C_{\RBO}^\bullet(\frakg)$.
\end{prop}
\begin{remark} The above result shows  that the cochain complex $\C^\bullet_\RBO(\frakg)$ controls formal deformations of Rota-Baxter operators,   which justifies the name ``cochain complex of Rota-Baxter operators".
In fact, we were inspired by  Equation~\eqref{Eq: deform eq1 for RBO} while defining $\C^\bullet_\RBO(\frakg)$.\end{remark}

\subsection{Formal deformations of Rota-Baxter $3$-Lie algebras}\

Let $(\frakg,\mu, T)$ be a  Rota-Baxter $3$-Lie algebra of  weight $\lambda$.   Consider a 1-parameterized family:
\[\mu_t=\sum_{i=0}^\infty \mu_it^i, \ \mu_i\in \C^2_\PLA(\frakg),\quad  T_t=\sum_{i=0}^\infty T_it^i,  \ T_i\in \C^1_{\RBO}(\frakg).\]

\begin{defn}
	A  \textbf{1-parameter formal deformation of   Rota-Baxter $3$-Lie algebra}  $(\frakg, \mu,T)$ of  weight $\lambda$ is a pair $(\mu_t,T_t)$ which endows the flat $\bfk[[t]]$-module $\frakg[[t]]$ with a Rota-Baxter $3$-Lie algebra structure of weight $\lambda$ over $\bfk[[t]]$ such that $(\mu_0,T_0)=(\mu,T)$. The pair $(\mu_1,T_1)$ is called the \textbf{infinitesimal} of the 1-parameter formal deformation $(\frakg[[t]],\mu_t,T_t)$ of   Rota-Baxter $3$-Lie algebra $(\frakg,\mu,T)$ of  weight $\lambda$.
\end{defn}

 Power series $\mu_t$ and $ T_t$ determine a  1-parameter formal deformation of  Rota-Baxter $3$-Lie algebra  $(\frakg,\mu,T)$ of   weight $\lambda$ if and only if for arbitrary  $x_1, x_2, x_3, x_4, x_5\in \frakg$, the following equations hold :
\begin{align*}
&\mu_t(x_1,x_2, \mu_t(x_3,x_4, x_5))\\
=\quad & \mu_t(\mu_t(x_1,x_2, x_3),x_4, x_5)+\mu_t(x_3,\mu_t(x_1, x_2, x_4), x_5)+\mu_t(x_3,x_4, \mu_t(x_1,x_2, x_5))\\
& \mu_t(T_t(x_1),  T_t(x_2), T_t(x_3))\\
 =\quad &T_t(\mu_t(T_t(x_1), T_t(x_2), x_3)+\mu_t(T_t(x_1), x_2, T_t(x_3))+\mu_t(x_1, T_t(x_2), T_t(x_3))\\
&+\lambda \  \mu_t(T_t(x_1), x_2, x_3)+\lambda \  \mu_t(x_1, T_t(x_2), x_3)+\lambda \  \mu_t(x_1, x_2, T_t(x_3))\\
&+\lambda^2  \mu_t(x_1, x_2, x_3)\big).
 \end{align*}
By expanding these equations and comparing the coefficient of $t^n$, we obtain  that $\{\mu_i\}_{i\geqslant0}$ and $\{T_i\}_{i\geqslant0}$ have to  satisfy: for arbitrary  $n\geqslant 0$,
\begin{align}\label{Eq: deform eq for  products in RBA}
&\sum_{i+j=n}\mu_i(x_1,x_2, \mu_j(x_3,x_4, x_5))\nonumber\\
= &\sum_{i+j=n}\mu_i(\mu_j(x_1,x_2, x_3),x_4, x_5)+\mu_i(x_3,\mu_j(x_1, x_2, x_4), x_5)+\mu_i(x_3,x_4, \mu_j(x_1,x_2, x_5)),\end{align}
 \begin{eqnarray}\label{Eq: Deform RB operator in RBA}
 	 	&&\sum_{i+j+k+l=n}\mu_i(T_j(x_1),  T_k(x_2), T_l(x_3))\nonumber\\
&=&\sum_{i+j+k+l=n}T_i\big(\mu_j(T_k(x_1), T_l(x_2), x_3)+\mu_j(T_k(x_1), x_2, T_l(x_3))+\mu_j(x_1, T_k(x_2), T_l(x_3))\big)\\
&&+\lambda \ \sum_{i+j+k=n} T_i\big(\mu_j(T_k(x_1), x_2, x_3)+  \mu_j(x_1, T_k(x_2), x_3)+  \mu_j(x_1, x_2, T_k(x_3))\big)\nonumber\\
&&+\lambda^2 \sum_{i+j=n} T_i\big( \mu_j(x_1, x_2, x_3)\big).\nonumber
\end{eqnarray}
Obviously, when $n=0$, the above conditions are exactly the 3-Lie bracket  $\mu=\mu_0$ and Equation~(\ref{Eq: Rota-Baxter relation}) which is the defining relation of  Rota-Baxter operator $T=T_0$ of  weight $\lambda$.

\smallskip

\begin{prop}\label{Prop: Infinitesimal is 2-cocyle}
	Let $(\frakg[[t]],\mu_t,T_t)$ be a  1-parameter formal deformation of
	Rota-Baxter $3$-Lie algebra  $(\frakg,\mu,T)$ of  weight $\lambda$. Then
	$(\mu_1,T_1)$ is a 2-cocycle in the cochain complex
	$C_{\RBA}^\bullet(\frakg)$.
\end{prop}
\begin{proof} When $n=1$,   Equations~(\ref{Eq: deform eq for  products in RBA}) and (\ref{Eq: Deform RB operator in RBA})  become
	 \begin{align*}
&[x_1,x_2, \mu_1(x_3,x_4, x_5)]_\frakg+\mu_1(x_1,x_2, [x_3,x_4, x_5]_\frakg)\\
=\quad &[\mu_1(x_1,x_2, x_3),x_4, x_5]_\frakg+[x_3,\mu_1(x_1, x_2, x_4), x_5]_\frakg+[x_3,x_4, \mu_1(x_1,x_2, x_5)]_\frakg\\
&+\mu_1([x_1,x_2, x_3]_\frakg,x_4, x_5)+\mu_1(x_3,[x_1, x_2, x_4]_\frakg, x_5)+\mu_1(x_3,x_4, [x_1,x_2, x_5]_\frakg)
\end{align*}
and
\begin{eqnarray}\label{Eq: Deform RB operator 1 in RBA}
&&\\
&&\mu_1(T(x_1), T(x_2), T(x_3))-T\big(\mu_1(T(x_1), T(x_2), x_3)-\mu_1(T(x_1), x_2, T(x_3))-\mu_1(x_1, T(x_2), T(x_3))\big)\nonumber\\
&&-\lambda \  T\big(\mu_1(T(x_1), x_2, x_3)+  \mu_1(x_1, T(x_2), x_3)+  \mu_1(x_1, x_2, T(x_3))\big)-\lambda^2  T\big(\mu_1(x_1, x_2, x_3)\big)\nonumber\\
&=&T_1\big([T(x_1), T(x_2), x_3]_\frakg+[T(x_1), x_2, T(x_3)]_\frakg+[x_1, T(x_2), T(x_3)]_\frakg\big)\nonumber\\
&&+\lambda \  T_1\big([T(x_1), x_2, x_3]_\frakg+  [x_1, T(x_2), x_3]_\frakg+  [x_1, x_2, T(x_3)]_\frakg\big)+\lambda^2  T_1\big([x_1, x_2, x_3]_\frakg\big)\nonumber\\
&&-[T_1(x_1),  T(x_2), T(x_3)]_\frakg-[T(x_1),  T_1(x_2), T(x_3)]_\frakg-[T(x_1),  T(x_2), T_1(x_3)]_\frakg\nonumber\\
&&+T\big([T_1(x_1), T(x_2), x_3]_\frakg+[T_1(x_1), x_2, T(x_3)]_\frakg+[x_1, T_1(x_2), T(x_3)]_\frakg\big)\nonumber\\
&&+T\big([T(x_1), T_1(x_2), x_3]_\frakg+[T(x_1), x_2, T_1(x_3)]_\frakg+[x_1, T(x_2), T_1(x_3)]_\frakg\big)\nonumber\\
&&+\lambda \  T\big([T_1(x_1), x_2, x_3]_\frakg+  [x_1, T_1(x_2), x_3]_\frakg+  [x_1, x_2, T_1(x_3)]_\frakg\big).\nonumber
\end{eqnarray}
Note that  the first equation is exactly $\delta^2(\mu_1)=0\in \C^\bullet_{\PLA}(\frakg)$ and that  second equation is exactly to  \[\Phi^2(\mu_1)=-\partial^1(T_1) \in \C^\bullet_{\RBO}(\frakg).\]
	So $(\mu_1,T_1)$ is a 2-cocycle in $\C^\bullet_{\RBA}(\frakg)$.
	\end{proof}

\begin{remark} Equation~\eqref{Eq: Deform RB operator 1 in RBA} inspired us to introduce the chain map $\Phi^\bullet:\C^\bullet_{\PLA}(\frakg ,M) \rightarrow C_{\RBO}^\bullet(\frakg ,M)$ as well the cochain complex $\C^\bullet_{\RBA}(\frakg, M) $.

\end{remark}
\smallskip
\begin{defn}
Let $(\frakg[[t]],\mu_t,T_t)$ and $(\frakg[[t]],\mu_t',T_t')$ be two 1-parameter formal deformations of  Rota-Baxter $3$-Lie algebra  $(\frakg,\mu,T)$ of  weight $\lambda$. A formal isomorphism from $(\frakg[[t]],\mu_t',T_t')$ to $(\frakg[[t]], \mu_t, T_t)$ is a power series $\psi_t=\sum_{i=0}\psi_it^i: \frakg[[t]]\rightarrow \frakg[[t]]$, where $\psi_i: \frakg\rightarrow \frakg$ are linear maps with $\psi_0=\Id_\frakg$, such that:
\begin{eqnarray}\label{Eq: equivalent deformations}\psi_t\circ \mu_t' &=& \mu_t\circ (\psi_t\ot \psi_t\ot \psi_t),\\
\psi_t\circ T_t'&=&T_t\circ\psi_t. \label{Eq: equivalent deformations2}
	\end{eqnarray}
	In this case, we say that the two 1-parameter formal deformations $(\frakg[[t]], \mu_t,T_t)$ and
	$(\frakg[[t]],\mu_t',T_t')$ are  equivalent.
\end{defn}

\smallskip

Given a  Rota-Baxter $3$-Lie algebra  $(\frakg,\mu,T)$ of  weight $\lambda$, the power series $\mu_t,T_t$
with $\mu_i=\delta_{i,0}\mu, T_i=\delta_{i,0}T$ makes
$(\frakg[[t]],\mu_t,T_t)$ into a $1$-parameter formal deformation of
$(\frakg,\mu,T)$. Formal deformations is said
to be trivial if it is  equivalent to $(\frakg,\mu,T)$.
\smallskip

\begin{thm}
The infinitesimals of two equivalent 1-parameter formal deformations of $(\frakg,\mu,T)$ are in the same cohomology class in $\rmH^\bullet_{\RBA}(\frakg)$.
\end{thm}

\begin{proof} Let $\psi_t:(\frakg[[t]],\mu_t',T_t')\rightarrow (\frakg[[t]],\mu_t,T_t)$ be a formal isomorphism.
	Expanding the identities and collecting coefficients of $t$, we get from Equations~(\ref{Eq: equivalent deformations}) and (\ref{Eq: equivalent deformations2}):
	\begin{eqnarray*}
		\mu_1'&=&\mu_1+\mu\circ(\Id\ot\Id\ot \psi_1)-\psi_1\circ\mu+\mu\circ(\psi_1\ot \Id\ot\Id)+\mu\circ( \Id\ot\psi_1\ot\Id),\\
		T_1'&=&T_1+T\circ\psi_1-\psi_1\circ T,
		\end{eqnarray*}
	that is, we have\[(\mu_1',T_1')-(\mu_1,T_1)=(\delta^1(\psi_1), -\Phi^1(\psi_1))=d^1(\psi_1,0)\in  \C^\bullet_{\RBA}(\frakg).\]
\end{proof}

\smallskip

\begin{defn}
	A  Rota-Baxter $3$-Lie algebra  $(\frakg,\mu,T)$ of  weight $\lambda$ is said to be rigid if every 1-parameter formal deformation is trivial.
\end{defn}

\begin{thm}
	Let $(\frakg,\mu,T)$ be a  Rota-Baxter $3$-Lie algebra of  weight $\lambda$. If $\rmH^2_{\RBA}(\frakg)=0$, then $(\frakg,\mu,T)$ is rigid.
\end{thm}

\begin{proof}Let $(\frakg[[t]], \mu_t, T_t)$ be a $1$-parameter formal deformation of $(\frakg, \mu, T)$. By Proposition~\ref{Prop: Infinitesimal is 2-cocyle},
$(\mu_1, T_1)$ is a $2$-cocycle. By $\rmH^2_{\RBA}(\frakg)=0$, there exists a $1$-cochain $$(\psi_1', x) \in \C^1_\RBA(\frakg)= C^1_{\PLA}(\frakg)\oplus \Hom(k, \frakg)$$ such that
$(\mu_1, T_1) =  d^1(\psi_1', x), $
that is, $\mu_1=\delta^1(\psi_1')$ and $T_1=-\partial^0(x)-\Phi^1(\psi_1')$. Let $\psi_1=\psi_1'+\delta^0(x)$. Then
 $\mu_1= \delta^1(\psi_1)$ and $T_1=-\Phi^1(\psi_1)$, as it can be readily seen that $\Phi^1(\delta^0(x))=\partial^0(x)$.

Setting $\psi_t = \Id_\frakg -\psi_1t$, we have a deformation $(\frakg[[t]], \overline{\mu}_t, \overline{T}_t)$, where
$$\overline{\mu}_t=\psi_t^{-1}\circ \mu_t\circ (\psi_t\times \psi_t\times \psi_t)$$
and $$\overline{T}_t=\psi_t^{-1}\circ T_t\circ \psi_t.$$
  It is easy to  verify  that $\overline{\mu}_1=0, \overline{T}_1=0$. Then
    $$\begin{array}{rcl} \overline{\mu}_t&=& \mu+\overline{\mu}_2t^2+\cdots,\\
 T_t&=& T+\overline{T}_2t^2+\cdots.\end{array}$$
   By Equations~(\ref{Eq: deform eq for  products in RBA}) and (\ref{Eq: Deform RB operator in RBA}), we see that $(\overline{\mu}_2,  \overline{T}_2)$ is still a $2$-cocycle, so by induction, we can show that
  $ (\frakg[[t]], \mu_t , T_t) $ is equivalent to  $(\frakg[[t]], \mu, T).$
Thus, $(\frakg,\mu,T)$ is rigid.

\end{proof}

\bigskip

\section{Abelian extensions of  Rota-Baxter $3$-Lie algebras}

 In this section, we study abelian extensions of  Rota-Baxter $3$-Lie algebras of  arbitrary weights  and show that they are classified by the second cohomology, as one would expect of a good cohomology theory.

 Notice that a vector space $M$ together with a linear transformation $T_M:M\to M$ is naturally a  Rota-Baxter $3$-Lie algebra of weight $\lambda$ where the bracket on $M$ is defined to be $[m, n, p]_M=0$ for all $m, n, p\in M.$

 \begin{defn}
 	An   \textbf{abelian extension  of  Rota-Baxter $3$-Lie algebras} of weight $\lambda$ is a short exact sequence of  morphisms of  Rota-Baxter $3$-Lie algebras of  weight $\lambda$
 \begin{eqnarray}\label{Eq: abelian extension} 0\to (M,[\cdot, \cdot, \cdot]_M, T_M)\stackrel{i}{\to} (\hat{\frakg},[\cdot, \cdot, \cdot]_{\hat{\frakg}},  \hat{T})\stackrel{p}{\to} (\frakg, [\cdot, \cdot, \cdot]_{\frakg}, T)\to 0,
 \end{eqnarray}
 that is, there exists a commutative diagram:
 	\[\begin{CD}
 		0@>>> {M} @>i >> \hat{\frakg} @>p >> \frakg @>>>0\\
 		@. @V {T_M} VV @V {\hat{T}} VV @V T VV @.\\
 		0@>>> {M} @>i >> \hat{\frakg} @>p >> \frakg @>>>0,
 	\end{CD}\]
 where the  Rota-Baxter $3$-Lie algebra  $(\hat{\frakg}, [\cdot, \cdot, \cdot]_{\hat{\frakg}}, T_{\hat{\frakg}})$ of  weight $\lambda$	satisfies  $[\cdot ,n, p]_{\hat{\frakg}}=0$ for all $n, p \in M$.

 We will call $(\hat{\frakg},[\cdot, \cdot, \cdot]_{\hat{\frakg}}, \hat{T})$ an abelian extension of $(\frakg,[\cdot, \cdot, \cdot]_{\frakg}, T)$ by $(M,[\cdot, \cdot, \cdot]_M, T_M)$.
 \end{defn}

 \begin{defn}
 	Let $(\hat{\frakg}_1,[\cdot, \cdot, \cdot]_{\hat{\frakg}_1}, \hat{T}_1)$ and $(\hat{\frakg}_2,[\cdot, \cdot, \cdot]_{\hat{\frakg}_2}, \hat{T}_2)$ be two abelian extensions of $(\frakg,[\cdot, \cdot, \cdot]_{\frakg}, T)$ by $(M,[\cdot, \cdot, \cdot]_M, T_M)$. They are said to be  isomorphic  if there exists an isomorphism of  Rota-Baxter $3$-Lie algebras $\zeta:(\hat{\frakg}_1,\hat{T}_1)\rar (\hat{\frakg}_2,\hat{T}_2)$ of  weight $\lambda$ such that the following commutative diagram holds:
 	\begin{eqnarray}\label{Eq: isom of abelian extension}\begin{CD}
 		0@>>> {(M,[\cdot, \cdot, \cdot]_M, T_M)} @>i_1 >> (\hat{\frakg}_1,[\cdot, \cdot, \cdot]_{\hat{\frakg}_1},{\hat{T}_1}) @>p_1 >> (\frakg,[\cdot, \cdot, \cdot]_{\frakg}, T) @>>>0\\
 		@. @| @V \zeta VV @| @.\\
 		0@>>> {(M,[\cdot, \cdot, \cdot]_M, T_M)} @>i_2 >> (\hat{\frakg}_2,[\cdot, \cdot, \cdot]_{\hat{\frakg}_2},{\hat{T}_2}) @>p_2 >> (\frakg,[\cdot, \cdot, \cdot]_{\frakg}, T) @>>>0.
 	\end{CD}\end{eqnarray}
 \end{defn}

 A   section of an abelian extension $(\hat{\frakg},[\cdot, \cdot, \cdot]_{\hat{\frakg}}, {\hat{T}})$ of $(\frakg,[\cdot, \cdot, \cdot]_{\frakg}, T)$ by $(M,[\cdot, \cdot, \cdot]_M, T_M)$ is a linear map $s:\frakg\rar \hat{\frakg}$ such that $p\circ s=\Id_\frakg$.

 We will show that isomorphism classes of  abelian extensions of $(\frakg,[\cdot, \cdot, \cdot]_{\frakg}, T)$ by $(M,[\cdot, \cdot, \cdot]_M, T_M)$ are in bijection with the second cohomology group   ${\rmH}_{\RBA}^2(\frakg,M)$.

 \bigskip

Let    $(\hat{\frakg},[\cdot, \cdot, \cdot]_{\hat{\frakg}}, \hat{T})$ be  an abelian extension of $(\frakg,[\cdot, \cdot, \cdot]_{\frakg}, T)$ by $(M,[\cdot, \cdot, \cdot]_M, T_M)$. Choose a section $s:\frakg\rar \hat{\frakg}$. Define $\rho: \wedge^2\frakg\rightarrow \mathfrak{gl}(M)$
 \begin{eqnarray*}
\rho(x)m=\rho(x_1, x_2)m:=[s(x_1), s(x_2), i(m)]_{\hat{\frakg}}.
 \end{eqnarray*}
 for all $x = (x_1, x_2) \in \wedge^2\frakg, m\in M$.
 \begin{prop}\label{Prop: new RB bimodules from abelian extensions}
 	With the above notations, $(M, \rho, T_M)$ is a representation over $(\frakg,[\cdot, \cdot, \cdot]_\frakg, T)$.
 \end{prop}
 \begin{proof}
 	It is straightforward to check that $(M, \rho)$ is a representation over $(\frakg,[\cdot, \cdot, \cdot]_\frakg)$. Moreover, ${\hat{T}}(s(x))-s(T(x))\in M$ means that  $\rho({\hat{T}}(s(x)) m=\rho(s(T(x))) m$. Thus we have
 	\begin{align*}
 		&\rho(T(x_1),T(x_2))T_M(m)\\
= & \rho(sT(x_1),sT(x_2))T_M(m)\\
 		 = &\rho(\hat{T}(s(x_1)), \hat{T}(s(x_2)))T_M(m)\\
 		 = &T_M\big(\rho(\hat{T}(s(x_1)), \hat{T}(s(x_2))) m +\rho(\hat{T}(s(x_1)), s(x_2))T_M(m) +\rho(s(x_1), \hat{T}(s(x_2)))T_M(m) \\
&+  \lambda \ \rho(\hat{T}(s(x_1)), s(x_2)) m+\lambda \ \rho(s(x_1), \hat{T}(s(x_2))) m+\lambda \ \rho(s(x_1), s(x_2)) T_M(m)\\
&+\lambda^2\rho(s(x_1), s(x_2))m\big)\\
 		 = & T_M\big(\rho(T(x_1), T(x_2)) m +\rho(T(x_1), x_2)T_M(m) +\rho(x_1, T(x_2))T_M(m)\\
&+  \lambda \ \rho(T(x_1), x_2) m+\lambda \ \rho(x, T(x_2)) m+\lambda \ \rho(x_1, x_2) T_M(m)
 +\lambda^2\rho(x_1, x_2)m\big).
 	\end{align*}
 	Hence, $(M, \rho, T_M)$ is a  representation over $(\frakg, [\cdot, \cdot, \cdot]_\frakg,  T)$.
 \end{proof}

 We  further  define linear maps $\psi:\wedge^3\frakg\rar M$ and $\chi:\frakg\rar M$ respectively by
 \begin{align*}
 	\psi(x, y, z)&=[s(x), s(y), s(z)]_{\hat{\frakg}}-s([x, y, z]_\frakg),\quad\forall x, y, z\in \frakg,\\
 	\chi(x)&={\hat{T}}(s(x))-s(T(x)),\quad\forall x\in \frakg.
 \end{align*}

 \begin{prop}\label{prop:2-cocycle}
 	 The pair
 	$(\psi,\chi)$ is a 2-cocycle  of    Rota-Baxter $3$-Lie algebra $(\frakg,[\cdot, \cdot, \cdot]_\frakg, T)$ of  weight $\lambda$ with  coefficients  in the representation $(M,\rho, T_M)$ introduced in Proposition~\ref{Prop: new RB bimodules from abelian extensions}.
 \end{prop}
The proof is by direct computations, so it is left to the reader.

 The choice of the section $s$ in fact determines a splitting
 $$\xymatrix{0\ar[r]&  M\ar@<1ex>[r]^{i} &\hat{\frakg}\ar@<1ex>[r]^{p} \ar@<1ex>[l]^{t}& \frakg \ar@<1ex>[l]^{s} \ar[r] & 0}$$
 subject to $t\circ i=\Id_M, t\circ s=0$ and $ it+sp=\Id_{\hat{\frakg}}$.
 Then there is an induced isomorphism of vector spaces
 $$\left(\begin{array}{cc} p& t\end{array}\right): \hat{\frakg}\cong   \frakg\oplus M: \left(\begin{array}{c} s\\ i\end{array}\right).$$
We can  transfer the  Rota-Baxter $3$-Lie algebra structure of   weight $\lambda$ on $\hat{\frakg}$ to $\frakg\oplus M$ via this isomorphism.
  It is direct to verify that this  endows $\frakg\oplus M$ with a $3$-Lie bracket $[\cdot, \cdot, \cdot]_\psi$ and a Rota-Baxter operator $T_\chi$ defined by
 \begin{align}
 	\label{eq:mul}[x+m, y+n, z+p]_\psi&=[x, y, z]_\frakg+\rho(x, y)p+\rho(z, x)n+\rho(y, z)m+\psi(x, y, z),\\
 	\label{eq:dif}T_\chi(x+m)&=T(x)+\chi(x)+T_M(m),
 \end{align}
 for arbitrary   $x, y, z\in \frakg,\,m,n, p\in M$.
 Moreover, we get an abelian extension
 $$0\to (M, [\cdot, \cdot, \cdot]_M, T_M)\stackrel{\left(\begin{array}{cc} 0& \Id\end{array}\right) }{\longrightarrow} (\frakg\oplus M,[\cdot, \cdot, \cdot]_\psi, T_\chi)\stackrel{\left(\begin{array}{c} \Id\\ 0\end{array}\right)}{\longrightarrow} (\frakg, [\cdot, \cdot, \cdot]_\frakg, T)\longrightarrow 0$$
 which is easily seen to be  isomorphic to the original one \eqref{Eq: abelian extension}.

 \medskip

 Now we investigate the influence of different choices of   sections.

 \begin{prop}\label{prop: different sections give}
 \begin{itemize}
 \item[(a)] Different choices of the section $s$ give the same  representation on $(M, [\cdot, \cdot, \cdot]_M,$\\$ T_M)$;

 \item[(b)]   the cohomology class of $(\psi,\chi)$ does not depend on the choice of sections.

 \end{itemize}

 \end{prop}
 \begin{proof}Let $s_1$ and $s_2$ be two distinct sections of $p$.
  We define $\gamma:\frakg\rar M$ by $\gamma(x)=s_1(x)-s_2(x)$.

  Since the    Rota-Baxter $3$-Lie algebra v$(\hat{\frakg}, [\cdot, \cdot, \cdot]_{\hat{\frakg}}, T_{\hat{\frakg}})$ of  weight $\lambda$	satisfies   $[\cdot ,n, p]_{\hat{\frakg}}=0$ for all $n, p \in M,$, we have
  $$\rho(s_1(x), s_1(y))m=  \rho(s_2(x)+\gamma(x), s_2(y)+\gamma(y)) m=\rho(s_2(x),s_2(y) ) m.$$ So different choices of the section $s$ give the same  representation on $(M, [\cdot, \cdot, \cdot]_M, T_M)$;

  We   show that the cohomology class $[(\psi,\chi)]$ does not depend on the choice of sections.   Then
 	\begin{align*}
 		\psi_1(x, y, z)=&[s_1(x),  s_1(y), s_1(z)]_{\hat{\frakg}}-s_1([x, y, z]_\frakg)\\
 		=&[s_2(x)+\gamma(x), s_2(y)+\gamma(y), s_2(z)+\gamma(z)]_{\hat{\frakg}}-\big(s_2([x, y, z]_\frakg)+\gamma([x, y, z]_\frakg)\big)\\
 		 =&([s_2(x),  s_2(y), s_2(z)]_{\hat{\frakg}}-s_2([x, y, z]_\frakg))+[s_2(x), \gamma(y), \gamma(z)]_{\hat{\frakg}}+[s_2(x), s_2(y), \gamma(z)]_{\hat{\frakg}}\\
  &+[s_2(x), \gamma(y), s_2(z)]_{\hat{\frakg}}+[\gamma(x),  s_2(y), s_2(z)]_{\hat{\frakg}}+[\gamma(x),  s_2(y), \gamma(z)]_{\hat{\frakg}} \\
  &+[\gamma(x), \gamma(y), s_2(z)]_{\hat{\frakg}}-\gamma([x, y, z]_\frakg)\\
 		 =&([s_2(x),  s_2(y), s_2(z)]_{\hat{\frakg}}-s_2([x, y, z]_\frakg))+[x, \gamma(y), \gamma(z)]_{\hat{\frakg}}+ [x, y, \gamma(z)]_{\hat{\frakg}} \\
  &+ [x, \gamma(y), z]_{\hat{\frakg}} + [\gamma(x),  y, z]_{\hat{\frakg}} +[\gamma(x),  y, \gamma(z)]_{\hat{\frakg}}+[\gamma(x), \gamma(y), z]_{\hat{\frakg}}-\gamma([x, y, z]_\frakg)\\
 		 =&\psi_2(x, y, z)+\delta(\gamma)(x, y, z)
 	\end{align*}
 	and
 	\begin{align*}
 		\chi_1(x)&={\hat{T}}(s_1(x))-s_1(T(x))\\
 		&={\hat{T}}(s_2(x)+\gamma(x))-(s_2(T(x))+\gamma(T(x)))\\
 		&=({\hat{T}}(s_2(x))-s_2(T(x)))+{\hat{T}}(\gamma(x))-\gamma(T(x))\\
 		&=\chi_2(x)+T_M(\gamma(x))-\gamma(T(x))\\
 		&=\chi_2(x)-\Phi^1(\gamma)(x).
 	\end{align*}
 	That is, $(\psi_1,\chi_1)=(\psi_2,\chi_2)+d^1(\gamma)$. Thus $(\psi_1,\chi_1)$ and $(\psi_2,\chi_2)$ correspond to the same cohomology class  {in $\rmH_{\RBA}^2(\frakg,M)$}.

 \end{proof}

 We show now the isomorphic abelian extensions give rise to the same cohomology class.
 \begin{prop}Let $M$ be a vector space and  $T_M\in\End_\bfk(M)$. Then $(M,[\cdot, \cdot, \cdot]_M,  T_M)$ is a    Rota-Baxter $3$-Lie algebra with trivial $3$-Lie bracket of weight $\lambda$.
 Let $(\frakg,[\cdot, \cdot, \cdot]_\frakg, T)$ be a    Rota-Baxter $3$-Lie algebra of  weight $\lambda$.
 Two isomorphic abelian extensions of   Rota-Baxter $3$-Lie algebra $(\frakg,[\cdot, \cdot, \cdot]_\frakg,  T)$ of  weight $\lambda$ by  $(M, [\cdot, \cdot, \cdot]_M,  T_M)$  give rise to the same cohomology class  in $\rmH_{\RBA}^2(\frakg,M)$.
 \end{prop}
 \begin{proof}
  Assume that $(\hat{\frakg}_1, [\cdot, \cdot, \cdot]_{\hat{\frakg}_1},  {\hat{T}_1})$ and $(\hat{\frakg}_2, [\cdot, \cdot, \cdot]_{\hat{\frakg}_2}, {\hat{T}_2})$ are two isomorphic abelian extensions of $(\frakg,[\cdot, \cdot, \cdot]_{\frakg}, T)$ by $(M,[\cdot, \cdot, \cdot]_M, T_M)$ as is given in Equation~\eqref{Eq: isom of abelian extension}. Let $s_1$ be a section of $(\hat{\frakg}_1,[\cdot, \cdot, \cdot]_{\hat{\frakg}_1},  {\hat{T}_1})$. As $p_2\circ\zeta=p_1$, we have
 	\[p_2\circ(\zeta\circ s_1)=p_1\circ s_1=\Id_{\frakg}.\]
 	Therefore, $\zeta\circ s_1$ is a section of $(\hat{\frakg}_2,[\cdot, \cdot, \cdot]_{\hat{\frakg}_2}, {\hat{T}_2})$. Denote $s_2:=\zeta\circ s_1$. Since $\zeta$ is a homomorphism of   Rota-Baxter $3$-Lie algebras of  weight $\lambda$ such that $\zeta|_M=\Id_M$, $\zeta(\rho(x, y) m)=\zeta(\rho(s_1(x), s_1(y))m)=\rho(s_2(x), s_2(y)) m=\rho(x,y )m$, so $\zeta|_M: M\to M$ is compatible with the induced  representation.
 We have
 	\begin{align*}
 		\psi_2(x, y, z)&=[s_2(x), s_2(y),  s_2(z)]_{\hat{\frakg}_2}-s_2[x, y, z]_{\frakg}\\
 &=[\zeta(s_1(x)), \zeta(s_1(y)),  \zeta( s_1(z))]_{\hat{\frakg}_2}-\zeta(s_1([x, y, z]_{\frakg})\\
 		&=\zeta([s_1(x), s_1(y),  s_1(z)]_{\hat{\frakg}_1}-s_1([x, y, z]_{\frakg})\\
 &=\zeta(\psi_1(x, y, z))\\
 		&=\psi_1(x, y, z)
 	\end{align*}
 	and
 	\begin{align*}
 		\chi_2(x)&={\hat{T}_2}(s_2(x))-s_2(T(x))={\hat{T}_2}(\zeta(s_1(x)))-\zeta(s_1(T(x)))\\
 		&=\zeta({\hat{T}_1}(s_1(x))-s_1(T(x)))=\zeta(\chi_1(x))\\
 		&=\chi_1(x).
 	\end{align*}
 	Consequently, two isomorphic abelian extensions give rise to the same element in {$\rmH_{\RBA}^2(\frakg,M)$}.
\end{proof}
 \bigskip

 Now we consider the reverse direction.

 Let $(M, \rho, T_M)$ be a representation over  Rota-Baxter $3$-Lie algebra $(\frakg, [\cdot, \cdot, \cdot]_\frakg,  T)$ of  weight $\lambda$. Given two linear maps  $\psi:\wedge^3\frakg\rar M$ and $\chi:\frakg\rar M$, one can define  a $3$-Lie bracket $[\cdot, \cdot, \cdot]_\psi$ and an operator $T_\chi$  on  $\frakg\oplus M$ by Equations~(\ref{eq:mul}) and (\ref{eq:dif}).
 The following fact is important:
 \begin{prop}\label{prop:2-cocycle}
 	The triple $(\frakg\oplus M,[\cdot, \cdot, \cdot]_\psi, T_\chi)$ is a  Rota-Baxter $3$-Lie algebra of  weight $\lambda$   if and only if
 	$(\psi,\chi)$ is a 2-cocycle in the cochain complex of the  Rota-Baxter $3$-Lie algebra  $(\frakg,[\cdot, \cdot, \cdot]_\frakg, T)$ of  weight $\lambda$ with  coefficients  in $(M,\rho, T_M)$.
 In this case,    we obtain an abelian extension
  $$0\to (M, [\cdot, \cdot, \cdot]_M, T_M)\stackrel{\left(\begin{array}{cc} 0& \Id\end{array}\right) }{\longrightarrow} (\frakg\oplus M,[\cdot, \cdot, \cdot]_\psi, T_\chi)\stackrel{\left(\begin{array}{c} \Id\\ 0\end{array}\right)}{\longrightarrow} (\frakg, [\cdot, \cdot, \cdot]_\frakg, T)\to 0$$
  and the canonical section $s=\left(\begin{array}{cc}   \Id  & 0\end{array}\right): (\frakg, T)\to (\frakg\oplus M, T_\chi)$ endows $M$ with the original representation.
 \end{prop}
 \begin{proof}
 	If $(\frakg\oplus M,[\cdot, \cdot, \cdot]_\psi,T_\chi)$ is a   Rota-Baxter $3$-Lie algebra of  weight $\lambda$, then the condition that $[\cdot, \cdot, \cdot]_\psi$ is a $3$-Lie bracket implies
 	\begin{align*}\label{eq:mc}
 		&\psi(x_1, x_2, [x_3, x_4, x_5]_\frakg) + \rho(x_1, x_2)\psi(x_3, x_4, x_5)\\
=&\psi([x_1, x_2, x_3]_\frakg, x_4, x_5) + \psi(y_1, [x_1, x_2, x_4]_\frakg, x_5)+\psi(x_3, x_4, [x_1, x_2, x_5]_\frakg) \\
&+\rho(x_4, x_5)\psi(x_1, x_2, x_3)+\rho(x_5, x_3)\psi(x_1, x_2, x_4)+\rho(x_3, x_4)\psi(x_1, x_2, x_5),\nonumber
 	\end{align*}
 	which means $\delta^2(\psi)=0$ in $\C^\bullet_{\tLie}(\frakg,M)$.
 	Since $T_\chi$ is a  Rota-Baxter operator of  weight $\lambda$,
 	for arbitrary  $x, y, z\in \frakg, m,n, p\in M$, we have
 	\begin{eqnarray*}
&&[T_\chi(x+m),  T_\chi(y+n), T_\chi(z+p)]_\psi\\
&=&T_\chi\big([T_\chi(x+m), T_\chi(y+n), z+p]_\psi+[T_\chi(x+m), y+n, T_\chi(z+p)]_\psi\\
&&+[x+m, T_\chi(y+n), T_\chi(z+p)]_\psi+\lambda \ [T_\chi(x+m), y+n, z+p]_\psi+\lambda \ [x+m, T_\chi(y+n), z+p]_\psi\\
&&+\lambda \ [x+m, y+n, T_\chi(z+p)]_\psi+\lambda^2 [x+m, y+n, z+p]_\psi\big)\end{eqnarray*}	
 	Then $\chi,\psi$ satisfy the following equation:
 	\begin{align*}
 		&\rho(T(x), T(y))\chi(z)+\rho(T(z), T(x))\chi(y)+\rho(T(y), T(z))\chi(x)+\psi(T(x), T(y), T(z))\\
 		=\quad & T_M(\rho(T(y), z)\chi(x)+\rho(z, T(x))\chi(y)+\psi(T(x), T(y), z))+\chi([T(x), T(y), z]_\frakg)\\
 &+ T_M(\rho(T(x), y)\chi(z)+\rho(y, T(z))\chi(x)+\psi(T(x), y, T(z)))+\chi([T(x), y, T(z)]_\frakg)\\
 &+ T_M(\rho(x, T(y))\chi(z)+\rho(T(z), x)\chi(y)+\psi(x, T(y), T(z)))+\chi([x, T(y), T(z)]_\frakg)\\
 &+\lambda \ T_M(\rho(y, z)\chi(x)+\psi(T(x), y, z))+\lambda \ T_M(\rho(z, x)\chi(y)+\psi(x, T(y), z))\\
 &+\lambda\chi([x, T(y), z]_\frakg)+\lambda \ T_M(\rho(x, y)\chi(z)+\psi(x, y, T(z)))+\lambda\chi([T(x), y, z]_\frakg)\\
  &+\lambda\chi([x, y, T(z)]_\frakg)+ \lambda^2 \chi([x, y, z]_\frakg)+\lambda^2T_M(\psi(x, y, z)).
 	\end{align*}
 	
 	That is,
 	\[ \partial^1( \chi )+\Phi^2(\psi)=0.\]
 	Hence, $(\psi,\chi)$ is a  2-cocycle in $C^\bullet_{\RBA}(\frakg,M)$.
 	
 	 Conversely, if $(\psi,\chi)$ is a 2-cocycle, one can easily check that $(\frakg\oplus M,[\cdot, \cdot, \cdot]_\psi, T_\chi)$ is a  Rota-Baxter $3$-Lie algebra of weight $\lambda$.

  The last statement is clear.
 \end{proof}

 Finally, we show the following result:
 \begin{prop}
 	Two cohomologous $2$-cocycles give rise to isomorphic abelian extensions.
 \end{prop}
 \begin{proof}

 	  Given two 2-cocycles $(\psi_1,\chi_1)$ and $(\psi_2,\chi_2)$, we can construct two abelian extensions $(\frakg\oplus M,[\cdot, \cdot, \cdot]_{\psi_1},T_{\chi_1})$ and  $(\frakg\oplus M,[\cdot, \cdot, \cdot]_{\psi_2},T_{\chi_2})$ via Equations~\eqref{eq:mul} and \eqref{eq:dif}. If they represent the same cohomology  class {in $\rmH_{\RBA}^2(\frakg,M)$}, then there exists two linear maps $\gamma_0:k\rightarrow M, \gamma_1:\frakg\to M$ such that $$(\psi_1,\chi_1)=(\psi_2,\chi_2)+(\delta^1(\gamma_1),-\Phi^1(\gamma_1)-\partial^0(\gamma_0)).$$
 	Notice that $\partial^0=\Phi^1\circ\delta^0$. Define $\gamma: \frakg\rightarrow M$ to be $\gamma_1+\delta^0(\gamma_0)$. Then $\gamma$ satisfies
 	\[(\psi_1,\chi_1)=(\psi_2,\chi_2)+(\delta^1(\gamma),-\Phi^1(\gamma)).\]
 	
 	Define $\zeta:\frakg\oplus M\rar \frakg\oplus M$ by
 	\[\zeta(x+m):=x-\gamma(x)+m.\]
 	Then $\zeta$ is an isomorphism of these two abelian extensions $(\frakg\oplus M,[\cdot, \cdot, \cdot]_{\psi_1},T_{\chi_1})$ and  $(\frakg\oplus M,[\cdot, \cdot, \cdot]_{\psi_2},T_{\chi_2})$.
 \end{proof}

\bigskip

\section{Skeletal  Rota-Baxter $3$-Lie $2$-algebras and  crossed modules}

In this section, we introduce the notion of  Rota-Baxter $3$-Lie  2-algebras of  arbitrary weights  and
show that skeletal  Rota-Baxter $3$-Lie algebras of  arbitrary weights are classified by 3-cocycles of  Rota-Baxter $3$-Lie
algebras of  arbitrary weights.

\subsection{Skeletal  Rota-Baxter $3$-Lie  $2$-algebras} \

In the following, we first recall the definition of  $3$-Lie  2-algebras from \cite{ZLS17}.
\begin{defn}
A \textbf{$3$-Lie  $2$-algebra} $\mathcal{G}=\mathfrak{g}_0\oplus \mathfrak{g}_1$ over $\mathbf{k}$ consists of the following data:
\begin{itemize}
\item[(a)]  a  linear map: $d: \mathfrak{g}_1\rightarrow \mathfrak{g}_0$,
\item[(b)]  completely skew-symmetric trilinear maps $l_3 : \mathfrak{g}_i \otimes \mathfrak{g}_j\otimes \mathfrak{g}_k \rightarrow \mathfrak{g}_{i+j+k}$, where $0 \leq i + j+k \leq 1$,

\item[(c)] a multilinear map $l_5 : (\mathfrak{g}_0\wedge \mathfrak{g}_0) \otimes (\mathfrak{g}_0\wedge \mathfrak{g}_0\wedge \mathfrak{g}_0 )\rightarrow \mathfrak{g}_1$,
    \end{itemize}
such that for all $x, y, x_i \in \mathfrak{g}_0, i=1,..., 7$ and $\alpha, \beta\in \mathfrak{g}_1$, the following equalities are satisfied:
\begin{eqnarray*}
&&(1)\quad dl_3(x, y, \alpha)=l_3(x, y, d(\alpha)),\\
&&(2)\quad l_3(\alpha, d(\beta), x)=l_3(d(\alpha), \beta,  x),\\
&&(3)\quad dl_5(x_1, x_2, x_3, x_4, x_5)=-l_3(x_1, x_2, l_3(x_3, x_4, x_5))+l_3(x_3, l_3(x_1, x_2, x_4), x_5)+\\
&&\quad \quad l_3(l_3(x_1, x_2, x_3), x_4, x_5)+l_3(x_3, x_4, l_3(x_1,x_2, x_5)),\\
&&(4)  \quad l_5(d(\alpha), x_2, x_3, x_4, x_5)=-l_3(\alpha, x_2, l_3(x_3, x_4, x_5))+l_3(x_3, l_3(\alpha, x_2, x_4), x_5)+\\
&&\quad \quad \quad l_3(l_3(\alpha, x_2, x_3), x_4, x_5)+l_3(x_3, x_4, l_3(\alpha,x_2, x_5)),\\
&& (5)  \quad  l_5(x_1, x_2, d(\alpha), x_4, x_5)=-l_3(x_1, x_2, l_3(\alpha, x_4, x_5))+l_3(\alpha, l_3(x_1, x_2, x_4), x_5)+\\
&&\quad \quad \quad l_3(l_3(x_1, x_2, \alpha), x_4, x_5)+l_3(\alpha, x_4, l_3(x_1,x_2, x_5)),\\
&&(6)\quad l_3(l_5(x_1, x_2, x_3, x_4, x_5), x_6, x_7)+l_3(x_5, l_5(x_1, x_2, x_3, x_4, x_6), x_7)\\
&&\quad\quad \quad +l_3(x_1, x_2, l_5(x_3, x_4, x_5, x_6, x_7))+l_3(x_5, x_6, l_5(x_1, x_2, x_3, x_4, x_7))\\
&&\quad \quad\quad +l_5(x_1, x_2, l_3(x_3, x_4, x_5), x_6, x_7)+l_5(x_1, x_2. x_5, l_3(x_3, x_4, x_6), x_7)\\
&&\quad \quad\quad +l_5(x_1, x_2, x_5, x_6, l_3(x_3, x_4, x_7))=l_3(x_3, x_4, l_5(x_1, x_2, x_5, x_6, x_7))\\
&&\quad\quad \quad +l_5(l_3(x_1, x_2, x_3), x_4, x_5, x_6, x_7)+l_5(x_3, l_3(x_1, x_2, x_4), x_5, x_6, x_7)\\
&&\quad \quad\quad + l_5(x_3, x_4, l_3(x_1, x_2, x_5), x_6, x_7)+l_5(x_3, x_4, x_5, l_3(x_1, x_2, x_6),  x_7)\\
&&\quad\quad \quad +l_5(x_1, x_2, x_3, x_4, l_3(x_5, x_6, x_7))+l_5(x_3, x_4, x_5, x_6, l_3(x_1, x_2, x_7)).
\end{eqnarray*}
\end{defn}
 A   $3$-Lie  2-algebra  is said to
be {\bf skeletal} if $d = 0$.    A   $3$-Lie  2-algebra is said to
be {\bf strict} if $l_5 = 0$.

Inspired by \cite{JS21}, we give the notion of a  Rota-Baxter $3$-Lie  2-algebra of  arbitrary weights.
\begin{defn}\label{Def: RBpreLie 2-algebra}
A  \textbf{Rota-Baxter $3$-Lie  2-algebra} of weight $\lambda$ consists of a $3$-Lie 2-algebra $\mathcal{G} = (\mathfrak{g}_0, \mathfrak{g}_1, d, l_3, l_5)$
and a  Rota-Baxter operator $\Theta = (T_0, T_1, T_2)$ on $\mathcal{G}$, where $T_0:\mathfrak{g}_0\rightarrow \mathfrak{g}_0, T_1: \mathfrak{g}_1\rightarrow \mathfrak{g}_1$ and a completely skew-symmetric trilinear map $T_2: \mathfrak{g}_0\otimes \mathfrak{g}_0\otimes\mathfrak{g}_0\rightarrow \mathfrak{g}_1$ satisfying  the following equalities:
\begin{small}
\begin{eqnarray*}
	(a)&& T_0\circ d=d\circ T_1\\
(b)&& T_0(l_3(T_0(x), T_0(y), z)+l_3(x, T_0(y), T_0(z))+l_3(T_0(x), y, T_0(z))\\
\quad &&+\lambda l_3(T_0(x), y, z)+\lambda l_3(x, T_0(y), z)+\lambda l_3(x, y, T_0(z)) +\lambda^2 l_3(x, y, z))
-l_3(T_0(x), T_0(y), T_0(z))\\
&=&dT_2(x, y, z);\\
(c)&&T_1(l_3(T_0(x), T_0(y), \alpha)+l_3(x, T_0(y), T_1(\alpha))+l_3(T_0(x), y, T_1(\alpha))+\lambda l_3(T_0(x), y, \alpha)+\lambda l_3(x, T_0(y), \alpha)\\
\quad &&+\lambda l_3(x, y, T_1(\alpha)) +\lambda^2 l_3(x, y, \alpha))-l_3(T_1(\alpha), T_0(x))\\
&=&T_2( x, y, d(\alpha));\\
(d)&& l_5(T_0(x_1),T_0(x_2), T_0(x_3), T_0(x_4), T_0(x_5))+l_3(T_2(x_1, x_2, x_3), T_0(x_4), T_0(x_5))\\
\quad &&  +l_3(T_0(x_3), T_2(x_1, x_2, x_4),T_0(x_5))+l_3(T_0(x_3), T_0(x_4), T_2(x_1, x_2, x_5))\\
\quad && +T_2(l_3(T_0(x_1), T_0(x_2), x_3)+l_3(x_1, T_0(x_2), T_0(x_3))+l_3(T_0(x_1), x_2, T_0(x_3)))\\
\quad &&+\lambda l_3(T_0(x_1), x_2, x_3+\lambda l_3(x_1, T_0(x_2), x_3)+\lambda l_3(x_1, x_2, T_0(x_3)) +\lambda^2 l_3(x_1, x_2, x_3), x_4, x_5)\\
\quad &&+ T_2(x_3,l_3(T_0(x_1), T_0(x_2), x_4)+l_3(x_1, T_0(x_2), T_0(x_4))+l_3(T_0(x_1), x_2, T_0(x_4))\\
\quad &&+\lambda l_3(T_0(x_1), x_2, x_4)+\lambda l_3(x_1, T_0(x_2), x_4)+\lambda l_3(x_1, x_2, T_0(x_4)) +\lambda^2 l_3(x_1, x_2, x_4),x_5 )\\
\quad &&+T_2(x_3, x_4, l_3(T_0(x_1), T_0(x_2), x_5)+l_3(x_1, T_0(x_2), T_0(x_5))+l_3(T_0(x_1), x_2, T_0(x_5))\\
\quad &&+\lambda l_3(T_0(x_1), x_2, x_5)+\lambda l_3(x_1, T_0(x_2), x_5)+\lambda l_3(x_1, x_2, T_0(x_5)) +\lambda^2 l_3(x_1, x_2, x_5))\\
\quad &=&l_3(T_0(x_1), T_0(x_2), T_2(x_3, x_4, x_5))+T_2(x_1, x_2, l_3(T_0(x_3), T_0(x_4), x_5)+l_3(x_3, T_0(x_4), T_0(x_5))\\
\quad &&+l_3(T_0(x_3), x_4, T_0(x_5))+\lambda l_3(T_0(x_3), x_4, x_5)+\lambda l_3(x_3, T_0(x_4), x_5)\\
\quad &&+\lambda l_3(x_3, x_4, T_0(x_5)) +\lambda^2 l_3(x_3, x_4, x_5))+ T_1(l_5(x_1, x_2, x_3, x_4, x_5)).
\end{eqnarray*}
\end{small}
\end{defn}
We will denote a  Rota-Baxter $3$-Lie  2-algebra of  weight $\lambda$ by $(\mathcal{G}, \Theta)$. A  Rota-Baxter $3$-Lie  2-algebra of  weight $\lambda$ is said to
be {\bf skeletal} if $d = 0$.    A  Rota-Baxter $3$-Lie  2-algebra of weight $\lambda$ is said to
be {\bf strict} if $l_5 = 0, T_2=0$.

Let $(\mathcal{G}, \Theta)$ be a skeletal  Rota-Baxter $3$-Lie  2-algebra of  weight $\lambda$. Then $\mathcal{G}$ is a  skeletal 3-Lie 2-algebra. Therefore, we have
\begin{eqnarray}
&&\quad 0=-l_3(x_1, x_2, l_3(x_3, x_4, x_5))+l_3(x_3, l_3(x_1, x_2, x_4), x_5)+l_3(l_3(x_1, x_2, x_3), x_4, x_5)\\
&&\quad \quad +l_3(x_3, x_4, l_3(x_1,x_2, x_5)),\nonumber\\
&&\quad 0=-l_3(\alpha, x_2, l_3(x_3, x_4, x_5))+l_3(x_3, l_3(\alpha, x_2, x_4), x_5)+l_3(l_3(\alpha, x_2, x_3), x_4, x_5)\\
&&\quad \quad +l_3(x_3, x_4, l_3(\alpha,x_2, x_5)),\nonumber\\
&&\quad 0=-l_3(x_1, x_2, l_3(\alpha, x_4, x_5))+l_3(\alpha, l_3(x_1, x_2, x_4), x_5)+l_3(l_3(x_1, x_2, \alpha), x_4, x_5)\\
&& \quad \quad +l_3(\alpha, x_4, l_3(x_1,x_2, x_5)),\nonumber\\
&&\quad  l_3(l_5(x_1, x_2, x_3, x_4, x_5), x_6, x_7)+l_3(x_5, l_5(x_1, x_2, x_3, x_4, x_6), x_7) \label{skeletal 3-Lie 2-algebra 4}\nonumber\\
&&\quad +l_3(x_1, x_2, l_5(x_3, x_4, x_5, x_6, x_7))+l_3(x_5, x_6, l_5(x_1, x_2, x_3, x_4, x_7)) \nonumber\\
&&\quad +l_5(x_1, x_2, l_3(x_3, x_4, x_5), x_6, x_7)+l_5(x_1, x_2. x_5, l_3(x_3, x_4, x_6), x_7)\nonumber\\
&&\quad  +l_5(x_1, x_2, x_5, x_6, l_3(x_3, x_4, x_7))\nonumber\\
&&=l_3(x_3, x_4, l_5(x_1, x_2, x_5, x_6, x_7))
 +l_5(l_3(x_1, x_2, x_3), x_4, x_5, x_6, x_7) \\
 &&\quad +l_5(x_3, l_3(x_1, x_2, x_4), x_5, x_6, x_7)  + l_5(x_3, x_4, l_3(x_1, x_2, x_5), x_6, x_7)\nonumber\\
&& \quad +l_5(x_3, x_4, x_5, l_3(x_1, x_2, x_6),  x_7) +l_5(x_1, x_2, x_3, x_4, l_3(x_5, x_6, x_7))\nonumber\\
&&\quad  +l_5(x_3, x_4, x_5, x_6, l_3(x_1, x_2, x_7)).\nonumber
\end{eqnarray}
for all $x_i \in \mathfrak{g}_0, i=1,..., 7$ and $\alpha\in \mathfrak{g}_1$.

\begin{thm}
There is a one-to-one correspondence between skeletal  Rota-Baxter $3$-Lie  2-algebras of  arbitrary weights and 3-cocycles of   Rota-Baxter $3$-Lie  algebra of  arbitrary weights.

\end{thm}
\begin{proof}
Let $(\mathcal{G}, \Theta)$ be a skeletal   Rota-Baxter $3$-Lie  2-algebra of  weight $\lambda$. Then $(\mathfrak{g}_0, l_3, T_0)$ is a  Rota-Baxter $3$-Lie  algebra of  weight $\lambda$. Define linear maps
$f: (\mathfrak{g}_0\wedge \mathfrak{g}_0) \otimes (\mathfrak{g}_0\wedge \mathfrak{g}_0\wedge \mathfrak{g}_0 )\rightarrow \mathfrak{g}_1$ and $\theta: \mathfrak{g}_0 \otimes \mathfrak{g}_0\otimes \mathfrak{g}_0\rightarrow \mathfrak{g}_1$ by
\begin{eqnarray*}
& f(x_1, x_2, x_3, x_4, x_5)=l_5(x_1, x_2, x_3, x_4, x_5),\\
& \theta(x, y, z)=T_2(x, y, z).
\end{eqnarray*}
By (c) in Definition \ref{Def: RBpreLie 2-algebra}, we have
\begin{eqnarray}
&& f(T_0(x_1),T_0(x_2), T_0(x_3), T_0(x_4), T_0(x_5))+l_3(T_2(x_1, x_2, x_3), T_0(x_4), T_0(x_5))\label{skeletal 3-Lie 2-algebra 5}\nonumber\\
\quad &&  +l_3(T_0(x_3), \theta(x_1, x_2, x_4),T_0(x_5))+l_3(T_0(x_3), T_0(x_4), \theta(x_1, x_2, x_5))\nonumber\\
\quad && +\theta(l_3(T_0(x_1), T_0(x_2), x_3)+l_3(x_1, T_0(x_2), T_0(x_3))+l_3(T_0(x_1), x_2, T_0(x_3))+\lambda l_3(T_0(x_1), x_2, x_3)\nonumber\\
\quad &&+\lambda l_3(x_1, T_0(x_2), x_3)+\lambda l_3(x_1, x_2, T_0(x_3)) +\lambda^2 l_3(x_1, x_2, x_3), x_4, x_5)\nonumber\\
\quad &&+ \theta(x_3,l_3(T_0(x_1), T_0(x_2), x_4)+l_3(x_1, T_0(x_2), T_0(x_4))+l_3(T_0(x_1), x_2, T_0(x_4)) \\
\quad &&+\lambda l_3(T_0(x_1), x_2, x_4)+\lambda l_3(x_1, T_0(x_2), x_4)+\lambda l_3(x_1, x_2, T_0(x_4)) +\lambda^2 l_3(x_1, x_2, x_4),x_5 )\nonumber\\
\quad &&+T_2(x_3, x_4, l_3(T_0(x_1), T_0(x_2), x_5)+l_3(x_1, T_0(x_2), T_0(x_5))+l_3(T_0(x_1), x_2, T_0(x_5))\nonumber\\
\quad &&+\lambda l_3(T_0(x_1), x_2, x_5)+\lambda l_3(x_1, T_0(x_2), x_5)+\lambda l_3(x_1, x_2, T_0(x_5)) +\lambda^2 l_3(x_1, x_2, x_5))\nonumber\\
\quad &&-l_3(T_0(x_1), T_0(x_2), T_2(x_3, x_4, x_5))-\theta(x_1, x_2, l_3(T_0(x_3), T_0(x_4), x_5)+l_3(x_3, T_0(x_4), T_0(x_5))\nonumber\\
\quad &&+l_3(T_0(x_3), x_4, T_0(x_5))+\lambda l_3(T_0(x_3), x_4, x_5)+\lambda l_3(x_3, T_0(x_4), x_5)\nonumber\\
\quad &&+\lambda l_3(x_3, x_4, T_0(x_5)) +\lambda^2 l_3(x_3, x_4, x_5))- T_1(f(x_1, x_2, x_3, x_4, x_5))=0.\nonumber
\end{eqnarray}
 By (\ref{skeletal 3-Lie 2-algebra 4}) and (\ref{skeletal 3-Lie 2-algebra 5}),  we have $d^n(f, \theta)=0$. Thus,  $(f, \theta)$ is a 3-cocycle of  Rota-Baxter $3$-Lie  algebras of  weight $\lambda$.

The proof of the other direction is similar, and  we are done.
\end{proof}

\subsection{Crossed modules} \

Now we turn to the study on strict  Rota-Baxter $3$-Lie  2-algebras. First we introduce the notion of crossed
modules of  Rota-Baxter $3$-Lie  algebras.

\begin{defn}
A \textbf{crossed module} of  Rota-Baxter $3$-Lie  algebras of weight $\lambda$ is a sextuple $$((\mathfrak{g}_0, [\cdot, \cdot, \cdot]_{\mathfrak{g}_0}),
 (\mathfrak{g}_1, [\cdot, \cdot, \cdot]_{\mathfrak{g}_1}), d, S, T_0, T_1),$$
where $(\mathfrak{g}_0, [\cdot, \cdot, \cdot]_{\mathfrak{g}_0}, T_0)$ is a   Rota-Baxter $3$-Lie  algebra of weight $\lambda$ and $(\mathfrak{g}_1, [\cdot, \cdot, \cdot]_{\mathfrak{g}_1})$ is a  3-Lie algebra, $d: \mathfrak{g}_1\rightarrow \mathfrak{g}_0$ is a homomorphism of 3-Lie algebras,   and $(\mathfrak{g}_1, S,  T_1)$ is a Rota-Baxter representation  over $(\mathfrak{g}_0, [\cdot, \cdot, \cdot]_{\mathfrak{g}_0}, T_0)$ such that for all $x, y \in \mathfrak{g}_0$ and $\alpha, \beta, \gamma \in \mathfrak{g}_1$, the following
equalities are satisfied:
\begin{eqnarray*}
  &&d(S(x, y))(\alpha) = [x, y,  d(\alpha)]_{\mathfrak{g}_0}, T_0\circ d=d\circ T_1,\\
  && S(d(\alpha), d(\beta))(\gamma)= [\alpha, \beta, \gamma]_{\mathfrak{g}_1}, S(x, d(\alpha))(\beta)=-S(x, d(\beta))(\alpha).
\end{eqnarray*}

\end{defn}
Let $(\mathcal{G}, \Theta)$ be a strict  Rota-Baxter $3$-Lie  2-algebra of  weight $\lambda$. Then $\mathcal{G}$ is a  strict 3-Lie 2-algebra. Therefore, we have
\begin{eqnarray}
&&~dl_3(x, y, \alpha)=l_3(x, y, d(\alpha)),\\
&&~l_3(\alpha, d(\beta), x)=l_3(d(\alpha), \beta,  x),\label{strict RBA 2} \\
&&0=-l_3(x_1, x_2, l_3(x_3, x_4, x_5))+l_3(x_3, l_3(x_1, x_2, x_4), x_5)+l_3(l_3(x_1, x_2, x_3), x_4, x_5)\\
&&\quad\quad +l_3(x_3, x_4, l_3(x_1,x_2, x_5)),\label{strict RBA 3}\nonumber\\
&& 0=-l_3(\alpha, x_2, l_3(x_3, x_4, x_5))+l_3(x_3, l_3(\alpha, x_2, x_4), x_5)+l_3(l_3(\alpha, x_2, x_3), x_4, x_5)\\
&&\quad \quad +l_3(x_3, x_4, l_3(\alpha,x_2, x_5)),\nonumber\\
&&0=-l_3(x_1, x_2, l_3(\alpha, x_4, x_5))+l_3(\alpha, l_3(x_1, x_2, x_4), x_5)+l_3(l_3(x_1, x_2, \alpha), x_4, x_5)\label{strict RBA 5}\\
&& \quad\quad +l_3(\alpha, x_4, l_3(x_1,x_2, x_5)),\nonumber
\end{eqnarray}
for all $x_i \in \mathfrak{g}_0, i=1,..., 7$ and $\alpha, \beta\in \mathfrak{g}_1$.
\begin{thm}
There is a one-to-one correspondence between strict  Rota-Baxter $3$-Lie  2-algebras of  weight $\lambda$ and crossed
modules of  Rota-Baxter $3$-Lie  algebras of weight $\lambda$.
\end{thm}
\begin{proof}
Let $(\mathcal{G}, \Theta)$ be a strict  Rota-Baxter $3$-Lie  2-algebra of  weight $\lambda$. We construct a crossed module of
strict  Rota-Baxter $3$-Lie  algebras of  weight $\lambda$ as follows. Obviously, $(\mathfrak{g}_0, [\cdot, \cdot, \cdot]_{\mathfrak{g}_0}, T_0)$ is a  Rota-Baxter $3$-Lie  algebra of weight $\lambda$. Define brackets  $[\cdot,\cdot,\cdot]_{\mathfrak{g}_0}, [\cdot, \cdot, \cdot]_{\mathfrak{g}_1}$ on $\mathfrak{g}_0, \mathfrak{g}_1$ respectively as:
\begin{align*}
[x, y, z]_{\mathfrak{g}_0}&=l_3(x, y, z), \\
[\alpha, \beta, \gamma]_{\mathfrak{g}_1}&=l_3(d(\alpha), d(\beta), \gamma  )=l_3(d(\alpha), \beta, d(\gamma)  )=l_3(\alpha, d(\beta), d(\gamma)  ).
\end{align*}
By conditions (\ref{strict RBA 2}) and  (\ref{strict RBA 3}). Then    $(\mathfrak{g}_1, [\cdot, \cdot, \cdot]_{\mathfrak{g}_1})$ is a 3-Lie algebra.  Obviously, we deduce that $d$ is a
homomorphism between 3-Lie algebras. Define $S: \mathfrak{g}_0\wedge \mathfrak{g}_0  \rightarrow gl(\mathfrak{g}_1)$ by
\begin{eqnarray}
S(x, y)(\alpha) = l_3(x, y, \alpha).\label{crossed module}
\end{eqnarray}
By conditions (b),  (c) and (\ref{strict RBA 5}), (\ref{crossed module}),  the triple $(\mathfrak{g}_1, S, T_1)$ is a  representation  over the Rota-Baxter $3$-Lie algebra $(\mathfrak{g}_0, [\cdot, \cdot, \cdot]_{\mathfrak{g}_0}, T_0)$.
Thus $((\mathfrak{g}_0, [\cdot, \cdot, \cdot]_{\mathfrak{g}_0}), (\mathfrak{g}_1, [\cdot, \cdot, \cdot]_{\mathfrak{g}_1}), d, S, T_0, T_1)$ is a crossed module of  Rota-Baxter $3$-Lie   algebras of  weight $\lambda$.

Conversely, a crossed module of  Rota-Baxter $3$-Lie  algebras $((\mathfrak{g}_0, [\cdot, \cdot, \cdot]_{\mathfrak{g}_0}), (\mathfrak{g}_1, [\cdot, \cdot, \cdot]_{\mathfrak{g}_1}), $d, S,$T_0, T_1)$ of  weight $\lambda$ gives rise to a strict  Rota-Baxter $3$-Lie  2-algebra $(\mathfrak{g}_0,\mathfrak{g}_1, d, l_3, l_5 = 0, T_0, T_1, T_2=0)$ of  weight $\lambda$, where $l_3 : \mathfrak{g}_i \otimes \mathfrak{g}_j \otimes\mathfrak{g}_k \rightarrow \mathfrak{g}_{i+j+k} , 0 \leq i + j+k \leq 1$ is given by
\begin{eqnarray*}
l_3(x, y, z) = [x, y, z]_{\mathfrak{g}_0}, l_3(x, y, \alpha)=S(x, y)(\alpha).
\end{eqnarray*}
Direct verification shows that $(\mathfrak{g}_0, \mathfrak{g}_1, d, l_3, l_5=0, T_0, T_1, T_2=0)$  is a strict  Rota-Baxter $3$-Lie  2-algebra of  weight $\lambda$.
\end{proof}
\bigskip

\section*{Appendix A: Proof of Proposition~\ref{Prop: Chain map Phi} }
	The purpose of this appendix is to prove Proposition \ref{Prop: Chain map Phi}.
	
	Firstly, let's introduce some notations we will use in the following proof. For an element  $$(\mathfrak{X}_1, \dots, \mathfrak{X}_{n }, x_{n+1})\in\underbrace{\wedge^2\mathfrak{g}\otimes \cdots \otimes \wedge^2\mathfrak{g}}_{n-1}\wedge \mathfrak{g}$$ with $\mathfrak{X}_i=x_i\wedge y_i$, we write it as an ordered sequence $(z_1,z_2,\dots,z_{2n},z_{2n+1})$, where $$z_1=x_1,\ z_2=y_1,\dots, z_{2k-1}=x_{k},\ z_{2k}=y_k,\dots, z_{2n-1}=x_n,\ z_{2n}=y_n,\ z_{2n+1}=x_{n+1}.$$
	For arbitrary permutation $\tau \in S_{2n+1}$, its action on the element $(z_1,\dots,z_{2n+1})$ is given as   $$\tau(z_{1},z_2,\dots,z_{2n},z_{2n+1})= (z_{\tau^{-1}(1)},z_{\tau^{-1}(2)},\dots,z_{\tau^{-1}(2n)},z_{\tau^{-1}(2n+1)}).$$
 Define an operation  $[ i-1,i,i+1 ]_\mathfrak{g}$ from $\underbrace{\wedge^2\mathfrak{g}\otimes \cdots \otimes \wedge^2\mathfrak{g}}_{n-1}\wedge \mathfrak{g}$ to $\underbrace{\wedge^2\mathfrak{g}\otimes \cdots \otimes \wedge^2\mathfrak{g}}_{n-2}\wedge \mathfrak{g}$ as $$[ i-1,i,i+1 ]_\mathfrak{g}\big(z_{1},z_2,\dots,z_{2n},z_{2n+1}\big)=\big(z_{1},z_2,\dots,z_{i-2},[ z_{i-1},z_{i },z_{i+1} ]_\mathfrak{g},z_{i+2},\dots ,z_{2n+1}\big).$$

 {\noindent{\bf{Proposition \ref{Prop: Chain map Phi}}}
 	The map $\Phi^\bullet: \C^\bullet_\PLA(\frakg,M)\rightarrow \C^\bullet_{\RBO}(\frakg,M)$ is a chain map.
}

	\begin{proof}
		Let $n\geqslant 1$. For arbitrary  $ f\in \C^n_{\PLA}(\frakg,M)$ and $\frak{X}_1\otimes \frak{X}_2\otimes \dots \otimes \frak{X}_{n}\wedge x_{n+1}\in (\frak{g}^{\wedge 2})^{\ot n}\ot \frak{g}^{\wedge 3}$ with $\frak{X}_i=x_i\wedge y_i, \forall 1\leqslant i\leqslant n$,
\begin{small}
		\begin{eqnarray*}
			&&\Big(\Phi^{n+1}\delta^n(f)\Big)\Big(\mathfrak{X}_1, ... ,\mathfrak{X}_{n}, x_{n+1}\Big)\\
			&=&\Big(\delta^n(f)\circ\left( T,\cdots,T,T\right)\Big) \Big(\mathfrak{X}_1, ..., \mathfrak{X}_{n}, x_{n+1}\Big) \\
			&&-\sum_{k=0}^{2n}\lambda^{2n-k}\sum_{1\leqslant i_1<i_2<\cdots< i_{k}\leqslant 2n+1}   \Big(T_M\circ \delta^n(f)\circ (\Id^{(i_1-1)} ,  T , \Id^{ (i_2-i_1-1)} , T, \cdots, T,\Id^{  (2n+1-i_{k})}  )\Big)\Big(\mathfrak{X}_1, ..., \mathfrak{X}_{n}, x_{n+1}\Big)\\
		    &= &\sum_{1\leq j< k\leq n}(-1)^{j}f\Big(T(x_1)\wedge T(y_1), \cdots, \mathfrak{\hat{X}}_j, \cdots, T(x_{k-1})\wedge T(y_{k-1}), \Big([T(x_j), T(y_j), T(x_k)]_\mathfrak{g}\wedge T(y_k)\\
		    &&+ T(x_k)\wedge [T(x_j), T(y_j), T(y_k)]_\mathfrak{g}\Big),  T(x_{k+1})\wedge T(y_{k+1}), \cdots,  T(x_{n})\wedge T(y_{n}), T(x_{n+1})\Big)\\
		    && +\sum^{n}_{j=1}(-1)^{j}f\Big( T(x_{1})\wedge T(y_{1}), \cdots, \mathfrak{\hat{X}}_j, \cdots,  T(x_{n})\wedge T(y_{n}), [T(x_j), T(y_j), T(x_{n+1})]_\mathfrak{g}\Big)\\
		    &&+ \sum^{n}_{j=1}(-1)^{j+1}\rho\Big(T(x_j), T(y_j)\Big)f\Big(T(x_1)\wedge T(y_1), \cdots, \mathfrak{\hat{X}}_j, \cdots, T(x_{n})\wedge T(y_{n}), T(x_{n+1})\Big)\\
		    &&+ (-1)^{n+1} \rho\Big(T(y_{n}), T(x_{n+1})\Big)f\Big(T(x_1)\wedge T(y_1),  \cdots, T(x_{n-1})\wedge T(y_{n-1}), T(x_{n}) \Big) \\
		    &&  +(-1)^{n+1} \rho\Big(T(x_{n+1}), T(x_{n})\Big)f\Big(T(x_1)\wedge T(y_1),  \cdots, T(x_{n-1})\wedge T(y_{n-1}), T(y_{n})\Big) \\
		   	&&-\sum_{k=0}^{2n}\lambda^{2n-k}\sum_{1\leqslant i_1<i_2<\cdots< i_{k}\leqslant 2n+1}\sum_{1\leq j< k\leq n}(-1)^{j} \Big( T_M\circ f\circ      [ 2k-3,2k-2,2k-1]_\mathfrak{g}\circ    \\
		    &&\ \     (2j-1,\cdots,2k-2)^{-2}\circ (\Id^{(i_1-1)} ,  T , \Id^{ (i_2-i_1-1)} , T, \cdots, T,\Id^{  (2n+1-i_{k})}  )\Big)\Big(\mathfrak{X}_1, \cdots,  \mathfrak{X}_{n}, x_{n+1}\Big)\\
		    &&-\sum_{k=0}^{2n}\lambda^{2n-k}\sum_{1\leqslant i_1<i_2<\cdots< i_{k}\leqslant 2n+1}\sum_{1\leq j< k\leq n}(-1)^{j} \Big( T_M\circ f\circ      [ 2k-2,2k-1,2k ]_\mathfrak{g}\circ   \\
		    && \ \     (2j-1,\cdots,2k-1)^{-2}\circ (\Id^{(i_1-1)} ,  T , \Id^{ (i_2-i_1-1)} , T, \cdots, T,\Id^{  (2n+1-i_{k})}  )\Big)\Big(\mathfrak{X}_1, \cdots,  \mathfrak{X}_{n}, x_{n+1}\Big)\\
			&&-\sum_{k=0}^{2n}\lambda^{2n-k}\sum_{1\leqslant i_1<i_2<\cdots< i_{k}\leqslant 2n+1} \sum^{n}_{j=1}(-1)^{j}\Big(T_M\circ f\circ   [ 2n-1,2n,2n+1 ]_\mathfrak{g}\circ  \\
			&&\ \  (2j-1,\cdots,2n)^{-2}\circ(\Id^{(i_1-1)} ,  T , \Id^{ (i_2-i_1-1)} , T, \cdots, T,\Id^{  (2n+1-i_{k})}  )\Big)\Big(\mathfrak{X}_1, \cdots,  \mathfrak{X}_{n}, x_{n+1}\Big) \\
			&&-\sum_{k=0}^{2n}\lambda^{2n-k}\sum_{1\leqslant i_1<i_2<\cdots< i_{k}\leqslant 2n+1} \sum^{n}_{j=1}(-1)^{j+1}\Big( T_M\left( \rho\cdot f \right) \circ (1,\cdots, 2j)^2\circ \\
				&&\ \  (\Id^{(i_1-1)} ,  T , \Id^{ (i_2-i_1-1)} , T, \cdots, T,\Id^{  (2n+1-i_{k})}  )\big)\Big(\mathfrak{X}_1, \cdots,  \mathfrak{X}_{n}, x_{n+1}\Big) \\
			&&-\sum_{k=0}^{2n}\lambda^{2n-k}\sum_{1\leqslant i_1<i_2<\cdots< i_{k}\leqslant 2n+1}(-1)^{n+1}\Big\{ T_M \circ\big(\rho\cdot f\big)\circ\Big( (1,\cdots,2n+1)+(1,\cdots,2n+1)\circ (1,\cdots,2n-1)\Big) \circ\\
				&&\ \  \big(\Id^{(i_1-1)} ,  T , \Id^{ (i_2-i_1-1)} , T, \cdots, T,\Id^{  (2n+1-i_{k})}\big) \Big\}\Big(\mathfrak{X}_1, \cdots,  \mathfrak{X}_{n}, x_{n+1}\Big)
		\end{eqnarray*}
\end{small}
	On the other hand, we have
\begin{small}
		\begin{align*}
		&\partial^n\Phi^{n}\big(f\big)(\mathfrak{X}_1, ... \mathfrak{X}_{n}, x_{n+1})\\
		=&\sum_{1\leq j< k\leq n}(-1)^{j}\Phi^{n}\big(f\big)\Big(\mathfrak{X}_1, \cdots, \mathfrak{\hat{X}}_j, \cdots, \mathfrak{X}_{k-1}, [x_j, y_j, x_k]_T\wedge y_k+ x_k\wedge [x_j, y_j, y_k]_T, \mathfrak{X}_{k+1}, \cdots, \mathfrak{X}_{n}, x_{n+1}\Big)\\
		& +\sum^{n}_{j=1}(-1)^{j}\Phi^{n}\big(f\big)\Big(\mathfrak{X}_1, \cdots, \mathfrak{\hat{X}}_j, \cdots, \mathfrak{X}_{n}, [x_j, y_j, x_{n+1}]_T\Big)\\
		&+ \sum^{n}_{j=1}(-1)^{j+1}\rho \Big(T(x_j), T(y_j)\Big)\Phi^{n}\big(f\big)\Big(\mathfrak{X}_1, \cdots, \mathfrak{\hat{X}}_j, \cdots, \mathfrak{X}_{n}, x_{n+1}\Big)\\
		&-\sum^{n}_{j=1}(-1)^{j+1} T_M\Big( \rho\big(T(x_j),  y_j \big)\Phi^{n}\big(f\big)\big(\mathfrak{X}_1, \cdots, \mathfrak{\hat{X}}_j, \cdots, \mathfrak{X}_{n}, x_{n+1}\big)\Big) \\
		&-\sum^{n}_{j=1}(-1)^{j+1} T_M\Big( \rho\big( x_j ,  T(y_j) \big) \Phi^{n}\big(f\big)\big(\mathfrak{X}_1, \cdots, \mathfrak{\hat{X}}_j, \cdots, \mathfrak{X}_{n}, x_{n+1}\big)\Big) \\
		&-\sum^{n}_{j=1}(-1)^{j+1}\lambda \ T_M\Big( \rho\big( x_j ,  y_j \big) \Phi^{n}\big(f\big)\big(\mathfrak{X}_1, \cdots, \mathfrak{\hat{X}}_j, \cdots, \mathfrak{X}_{n}, x_{n+1}\big)\Big) \\
		&+ (-1)^{n+1}\Big(  \rho\big(T(y_{n}), T(x_{n+1})\big)\Phi^{n}\big(f\big)\big(\mathfrak{X}_1,  \cdots, \mathfrak{X}_{n-1}, x_{n}\big)+\rho \big(T(x_{n+1}), T(x_{n})\big)\Phi^{n}\big(f\big)\big(\mathfrak{X}_1,  \cdots, \mathfrak{X}_{n-1}, y_{n}\big) \Big)  \\
		&- (-1)^{n+1} T_M\Big( \rho\big(T(y_{n}),  x_{n+1} \big)\Phi^{n}\big(f\big)\big(\mathfrak{X}_1,  \cdots, \mathfrak{X}_{n-1}, x_{n}\big)+\rho \big(T(x_{n+1}), x_{n}\big)\Phi^{n}\big(f\big)\big(\mathfrak{X}_1,  \cdots, \mathfrak{X}_{n-1}, y_{n}\big)\Big) \\
		&- (-1)^{n+1}T_M\Big(  \rho\big( y_{n} , T(x_{n+1})\big)\Phi^{n}\big(f\big)\big(\mathfrak{X}_1,  \cdots, \mathfrak{X}_{n-1}, x_{n}\big)+\rho \big(x_{n+1}, T(x_{n})\big)\Phi^{n}\big(f\big)\big(\mathfrak{X}_1,  \cdots, \mathfrak{X}_{n-1}, y_{n}\big)\Big) \\
		&- (-1)^{n+1}\lambda \ T_M\Big(  \rho\big( y_{n} ,  x_{n+1} \big)\Phi^{n}\big(f\big)\big(\mathfrak{X}_1,  \cdots, \mathfrak{X}_{n-1}, x_{n}\big)+\rho \big(x_{n+1}, x_{n})\Phi^{n}\big(f\big)\big(\mathfrak{X}_1,  \cdots, \mathfrak{X}_{n-1}, y_{n}\big)\Big) \\
		&\\
		=&\sum_{1\leq j< k\leq n}(-1)^{j}f\Big(T(x_1)\wedge T(y_1), \cdots, {\mathfrak{\hat{X}_j}}, \cdots, T(x_{k-1})\wedge T(y_{k-1}), [T(x_j), T(y_j), T(x_k)]_\mathfrak{g}\wedge T(y_k)\\
		&\ \ \ \ \ \ \ \ \ \ \ \  + T(x_k)\wedge [T(x_j), T(y_j), T(y_k)]_\mathfrak{g},  T(x_{k+1})\wedge T(y_{k+1}), \cdots,  T(x_{n})\wedge T(y_{n}), T(x_{n+1})\Big)\\
			&-\sum_{k=0}^{2n}\lambda^{2n-k}\sum_{0\leqslant i_1<i_2<\cdots< i_{k}\leqslant 2n+1}\sum_{1\leq j< k\leq n}(-1)^{j}  T_M\circ f\circ      [ 2k-3,2k-2,2k-1]_\mathfrak{g}\circ (2j-1,\cdots,2k-2)^{-2}\circ   \\
		&     \ \ \ \ \ \ \ \ \ \ \ \ \ \ \ (\Id^{(i_1-1)} ,  T , \Id^{ (i_2-i_1-1)} , T, \cdots, T,\Id^{  (2n+1-i_{k})}  )\circ(\mathfrak{X}_1, \cdots,  \mathfrak{X}_{n}, x_{n+1})\\
		&-\sum_{k=0}^{2n}\lambda^{2n-k}\sum_{1\leqslant i_1<i_2<\cdots< i_{k}\leqslant 2n+1}\sum_{1\leq j< k\leq n}(-1)^{j}  T_M\circ f\circ      [ 2k-2,2k-1,2k ]_\mathfrak{g}\circ (2j-1,\cdots,2k-1)^{-2}\circ  \\
		& \ \ \ \ \ \ \ \ \ \  \ \ \     (\Id^{(i_1-1)} ,  T ,
		\Id^{ (i_2-i_1-1)} , T, \cdots, T,\Id^{
			(2n+1-i_{k})}  )\circ(\mathfrak{X}_1, \cdots,
		\mathfrak{X}_{n}, x_{n+1})\\
	    & +\sum^{n}_{j=1}(-1)^{j}f( T(x_{1})\wedge T(y_{1}), \cdots, \mathfrak{\hat{X}}_j, \cdots,  T(x_{n})\wedge T(y_{n}), [T(x_j), T(y_j), T(x_{n+1})]_\mathfrak{g})\\
	   	&-\sum_{k=0}^{2n}\lambda^{2n-k}\sum_{1\leqslant i_1<i_2<\cdots< i_{k}\leqslant 2n+1} \sum^{n}_{j=1}(-1)^{j}T_M\circ f\circ   [ 2n-1,2n,2n+1 ]_\mathfrak{g}\circ (2j-1,\cdots,2n)^{-2}\circ \\
	   & \ \ \ \ \ \ \ \ \ \ \big(\Id^{(i_1-1)} ,  T , \Id^{ (i_2-i_1-1)} , T, \cdots, T,\Id^{  (2n+1-i_{k})}  \big)\Big(\mathfrak{X}_1, \cdots,  \mathfrak{X}_{n}, x_{n+1}\Big) \\
	    & + \sum^{n+1}_{j=1}(-1)^{j+1}\rho\Big(T(x_j), T(y_j)\Big)f\Big(T(x_1)\wedge T(y_1), \cdots, \mathfrak{\hat{X}}_j, \cdots, T(x_{n})\wedge T(y_{n}), T(x_{n+1})\Big)\\
	   	&-\sum_{k=0}^{2n-2}\lambda^{2n-2-k}\sum_{1\leqslant i_1<i_2<\cdots< i_{k}\leqslant 2n-1}   \rho\Big(T(x_j), T(y_j)\Big)\cdot T_M\circ f\circ\Big(\Id^{(i_1-1)} ,T,   \cdots, \Id^{  (2n-1-i_{k})}  \Big) \Big(\mathfrak{X}_1, \cdots, \mathfrak{\hat{X}}_j, \cdots, \mathfrak{X}_{n}, x_{n+1}\Big)   \\
	   	&-\sum^{n}_{j=1}(-1)^{j+1} T_M\Big( \rho\big(T(x_j),  y_j \big)f\big(T(x_1)\wedge T(y_1), \cdots, \mathfrak{\hat{X}}_j, \cdots, T(x_{n})\wedge T(y_{n}), T(x_{n+1})\big)\Big) \\
	   		&+\sum_{k=0}^{2n-2}\lambda^{2n-2-k}\sum_{1\leqslant i_1<i_2<\cdots< i_{k}\leqslant 2n-1} \sum^{n}_{j=1}(-1)^{j+1}\\
	   	& \ \ \ \  \ \ \ \  T_M\circ \Big( \rho\big(T(x_j), y_j\big)\cdot T_M\circ f\circ\big(\Id^{(i_1-1)} ,T,   \cdots, \Id^{  (2n-1-i_{k})}  \big) \big(\mathfrak{X}_1, \cdots, \mathfrak{\hat{X}}_j, \cdots, \mathfrak{X}_{n}, x_{n+1}\big)\Big)   \\
	   	&- \sum^{n}_{j=1}(-1)^{j+1}T_M\Big( \rho\big(x_j, T(y_j) \big)\cdot f\big(T(x_1)\wedge T(y_1), \cdots, \mathfrak{\hat{X}}_j, \cdots, T(x_{n})\wedge T(y_{n}), T(x_{n+1})\big)\Big) \\
	   	&+\sum_{k=0}^{2n-1}\lambda^{2n-2-k}\sum_{1\leqslant i_1<i_2<\cdots< i_{k}\leqslant 2n-1} \sum^{n}_{j=1}(-1)^{j+1}\\
	   	& T_M\circ\Big(  \rho\big(x_j, T(y_j)\big) \cdot \Big( T_M\circ f\circ \big(\Id^{(i_1-1)} ,T,   \cdots, \Id^{  (2n-1-i_{k-1})}  \big) \big(\mathfrak{X}_1, \cdots, \mathfrak{\hat{X}}_j, \cdots, \mathfrak{X}_{n}, x_{n+1}\big)\Big)  \Big) \\
	   	&+ \sum^{n}_{j=1}(-1)^{j+1}\lambda \ T_M \Big( \rho(x_j,  y_j )\cdot f\Big(T(x_1)\wedge T(y_1), \cdots, \mathfrak{\hat{X}}_j, \cdots, T(x_n)\wedge T(y_n), T(x_{n+1})\Big)\Big) \\
	   	&+\sum_{k=0}^{2n-2}\lambda^{2n -k-1}\sum_{1\leqslant i_1<i_2<\cdots< i_{k-1}\leqslant 2n-1} \sum^{n}_{j=1}(-1)^{j+1}\\
	   	& \ \ \ \ \ \ T_M\circ\Big( \rho(x_j, y_j )\cdot\Big( T_M\circ f\circ \big(\Id^{(i_1-1)} ,T,   \cdots, \Id^{  (2n-1-i_{k-1})}  \big) \big(\mathfrak{X}_1, \cdots, \mathfrak{\hat{X}}_j, \cdots, \mathfrak{X}_{n}, x_{n+1}\big)\Big)  \Big) \\
	   	&+  (-1)^{n+1}  \rho\big(T(y_n), T(x_{n+1})\big)\cdot f\Big(T(x_1)\wedge T(y_1), \cdots, T(x_{n-1})\wedge T(y_{n-1})\wedge T(x_{n})\Big)\\
	   	&+ (-1)^{n+1}\rho \big(T(x_{n+1}), T(x_n)\big)\cdot f\Big(T(x_1)\wedge T(y_1), \cdots, T(x_{n-1})\wedge T(y_{n-1})\wedge T(y_{n})\Big)   \\
	   	&- (-1)^{n+1} \sum_{k=0}^{2n-2}\lambda^{2n-2-k}\sum_{1\leqslant i_1<i_2<\cdots< i_{k}\leqslant 2n-1} \rho\big(T(y_n), T(x_{n+1})\big)\cdot\Big\{ T_M\circ f\circ\\
	   	&\big(\Id^{(i_1-1)} ,T,   \cdots, \Id^{  (2n-1-i_{k-1})}  \big)\big(\mathfrak{X}_1,  \cdots, \mathfrak{X}_{n-1}, x_{n}\big)\Big\}\\
	   	&- (-1)^{n+1} \sum_{k=0}^{2n-2}\lambda^{2n-2-k}\sum_{1\leqslant i_1<i_2<\cdots< i_{k}\leqslant 2n-1} \\
	   	&\rho \big(T(x_{n+1}), T(x_n)\big)\cdot\Big\{ T_M\circ f\Big(\Id^{(i_1-1)} ,T,   \cdots, \Id^{  (2n-1-i_{k})}  \Big)\Big(\mathfrak{X}_1,  \cdots, \mathfrak{X}_{n-1}, y_{n}\Big)\Big\}\\
	   	&- (-1)^{n+1} T_M \Big\{  \rho\big(T(y_n),  x_{n+1} \big)\cdot f\Big(T(x_1)\wedge T(y_1), \cdots, T(x_{n-1})\wedge T(y_{n-1})\wedge T(x_{n})\Big)\Big\} \\
	   	&- (-1)^{n+1}T_M \Big\{\big( \rho (T(x_{n+1}), x_n\big)\cdot f\Big(T(x_1)\wedge T(y_1), \cdots, T(x_{n-1})\wedge T(y_{n-1})\wedge T(y_{n})\Big)  \Big\}  \\
	   	&+(-1)^{n+1} \sum_{k=0}^{2n-2}\lambda^{2n-2-k}\sum_{1\leqslant i_1<i_2<\cdots< i_{k}\leqslant 2n-1} \\
	   	&T_M \Big\{ \rho\big(T(y_n),  x_{n+1} \big)\cdot \Big(T_M\circ f\circ \big(\Id^{(i_1-1)} ,T,   \cdots, \Id^{  (2n-1-i_{k})}  \big)\big(\mathfrak{X}_1,  \cdots, \mathfrak{X}_{n-1}, x_{n}\big)\Big)\Big\} \\
	   	&+(-1)^{n+1} \sum_{k=0}^{2n-2}\lambda^{2n-2-k}\sum_{1\leqslant i_1<i_2<\cdots< i_{k}\leqslant 2n-1} \\
	   	&T_M \Big\{ \rho \big(T(x_{n+1}), x_n \big)\cdot\Big( T_M\circ f\circ \big(\Id^{(i_1-1)} ,T,   \cdots, \Id^{  (2n-1-i_{k})}  \big) \big(\mathfrak{X}_1,  \cdots, \mathfrak{X}_{n-1}, y_{n}\big)\Big)\Big\}\\
	   		&- (-1)^{n+1} T_M \Big\{  \rho\big( y_n ,  T(x_{n+1}) \big) \cdot  f\Big(T(x_1)\wedge T(y_1), \cdots, T(x_{n-1})\wedge T(y_{n-1})\wedge T(x_{n})\Big)\Big\}  \\
	   	&- (-1)^{n+1}T_M \Big\{ \rho \big( x_{n+1} , T(x_n) \big)\cdot f\Big(T(x_1)\wedge T(y_1), \cdots, T(x_{n-1})\wedge T(y_{n-1})\wedge T(y_{n})\Big)  \Big\}  \\
	   	&+(-1)^{n+1} \sum_{k=0}^{2n-2}\lambda^{2n-2-k}\sum_{1\leqslant i_1<i_2<\cdots< i_{k}\leqslant 2n-1} \\
	   	&T_M \Big\{ \rho\big( y_n ,  T(x_{n+1}) \big)\cdot \Big(T_M\circ f\big (\Id^{(i_1-1)} ,T,   \cdots, \Id^{  (2n-1-i_{k})}  \big)\big(\mathfrak{X}_1,  \cdots, \mathfrak{X}_{n-1}, x_{n}\big)\Big)\Big\} \\
	   	&+(-1)^{n+1} \sum_{k=0}^{2n-2}\lambda^{2n-2-k}\sum_{1\leqslant i_1<i_2<\cdots< i_{k}\leqslant 2n-1} \\
	   	&T_M \Big\{ \rho \big( x_{n+1} , T(x_n) \big)\cdot\Big( T_M\circ f \circ \big(\Id^{(i_1-1)} ,T,   \cdots, \Id^{  (2n-1-i_{k})}  \big) \big(\mathfrak{X}_1,  \cdots, \mathfrak{X}_{n-1}, y_{n}\big)\Big)\Big\} \\
	   	&- (-1)^{n+1}\lambda \  T_M\Big\{  \rho\big( y_n ,   x_{n+1}  \big)\cdot f\Big(T(x_1)\wedge T(y_1), \cdots, T(x_{n-1})\wedge T(y_{n-1})\wedge T(x_{n})\Big)\Big\} \\
	   	&-  (-1)^{n+1}\lambda \ T_M \Big\{ \rho \big( x_{n+1} , x_n  \big)\cdot f\Big(T(x_1)\wedge T(y_1), \cdots, T(x_{n-1})\wedge T(y_{n-1})\wedge T(y_{n})\Big)  \Big\}  \\
	   	&+(-1)^{n+1} \sum_{k=0}^{2n-2}\lambda^{2n -k-1}\sum_{1\leqslant i_1<i_2<\cdots< i_{k}\leqslant 2n-2} \\
	   	&\ \ \ \ \ \  \ \ T_M \Big\{ \rho( y_n ,   x_{n+1}  )\cdot \Big(T_M\circ f\circ \big(\Id^{(i_1-1)} ,T,   \cdots, \Id^{  (2n-1-i_{k})}  \big) \big(\mathfrak{X}_1,  \cdots, \mathfrak{X}_{n-1}, x_{n}\big)\Big)\Big\} \\
	   	&+(-1)^{n+1} \sum_{k=0}^{2n-2}\lambda^{2n -k-1}\sum_{1\leqslant i_1<i_2<\cdots< i_{k}\leqslant 2n-2} \\
	   	&\  \ \ \ \ \ \ \ T_M \Big\{ \rho ( x_{n+1} ,  x_n  )\cdot \Big( T_M\circ f\circ \big(\Id^{(i_1-1)} ,T,   \cdots, \Id^{  (2n-1-i_{k})}  \big) \big(\mathfrak{X}_1,  \cdots, \mathfrak{X}_{n-1}, y_{n}\big)\Big)\Big\} \\
	   	=& \sum_{1\leq j< k\leq n}(-1)^{j}f\Big( T(x_1)\wedge T(y_1), \cdots, \mathfrak{\hat{X}}_j, \cdots, T(x_{k-1})\wedge T(y_{k-1}), [T(x_j), T(y_j), T(x_k)]_\mathfrak{g}\wedge T(y_k)\\
	   	&\ \ \ \ \ \ \ \ \ \ \ \ \ + T(x_k)\wedge [T(x_j), T(y_j), T(y_k)]_\mathfrak{g},  T(x_{k+1})\wedge T(y_{k+1}), \cdots,  T(x_{n})\wedge T(y_{n})\wedge T(x_{n+1})\Big)\\
	   	&-\sum_{k=0}^{2n}\lambda^{2n-k}\sum_{1\leqslant i_1<i_2<\cdots< i_{k}\leqslant 2n+1}\sum_{1\leq j< k\leq n}(-1)^{j}  T_M\circ f\circ  [2k-1,2k,2k+1]_\mathfrak{g}\circ    \\
	   &  \ \ \ \ \ \    \big(2j-1,\cdots,2k\big)^{-2}\circ\big(\Id^{(i_1-1)} ,  T , \Id^{ (i_2-i_1-1)} , T, \cdots, T,\Id^{  (2n+1-i_{k})}  \big)\Big(\mathfrak{X}_1, \cdots,  \mathfrak{X}_{n}, x_{n+1}\Big)\\
	   &-\sum_{k=0}^{2n}\lambda^{2n-k}\sum_{1\leqslant i_1<i_2<\cdots< i_{k}\leqslant 2n+1}\sum_{1\leq j< k\leq n}(-1)^{j}  T_M\circ f\circ      [ 2k,2k+1,2k+2 ]_\mathfrak{g}\circ (2j-1,\cdots,2k+1)^{-2}\circ  \\
	   &     \ \ \ \ \ \ \big(\Id^{(i_1-1)} ,  T , \Id^{ (i_2-i_1-1)} , T, \cdots, T,\Id^{  (2n+1-i_{k})}  \big)\Big(\mathfrak{X}_1, \cdots,  \mathfrak{X}_{n}, x_{n+1}\Big)\\
	   	&  +\sum^{n}_{j=1}(-1)^{j}f\Big( T(x_{1})\wedge T(y_{1}), \cdots, \mathfrak{\hat{X}}_j, \cdots,  T(x_{n})\wedge T(y_{n})\wedge  [T(x_j), T(y_j), T(x_{n+1})]_\mathfrak{g}\Big)\\
	   	&-\sum_{k=0}^{2n}\lambda^{2n-k}\sum_{1\leqslant i_1<i_2<\cdots< i_{k}\leqslant 2n+1}\sum_{1\leq j< k\leq n}(-1)^{j}  T_M\circ f\circ      [ 2k,2k+1,2k+2 ]_\mathfrak{g}\circ \big(2j-1,\cdots,2k+1\big)^{-2}\circ  \\
	   	&   \ \ \ \ \ \ \ \ \   \circ \big(\Id^{(i_1-1)} ,  T , \Id^{ (i_2-i_1-1)} , T, \cdots, T,\Id^{  (2n+1-i_{k})}  \big)\Big(\mathfrak{X}_1, \cdots,  \mathfrak{X}_{n}, x_{n+1}\Big)\\
        &+ \sum^{n}_{j=1}(-1)^{j+1}\rho\big(T(x_j), T(y_j)\big)\cdot f\Big(T(x_1)\wedge T(y_1), \cdots, \mathfrak{\hat{X}}_j, \cdots, T(x_n)\wedge T(y_n)\wedge T(x_{n+1})\Big)\\
       	&-\sum_{k=0}^{2n}\lambda^{2n-k}\sum_{1\leqslant i_1<i_2<\cdots< i_{k}\leqslant 2n+1} \sum^{n}_{j=1}(-1)^{j+1} T_M\circ(\rho\cdot f) \circ (1,\cdots, 2j)^2\circ \\
       &\  \ \ \ \ \ \ \ \ \ \ \  \big(\Id^{(i_1-1)} ,  T , \Id^{ (i_2-i_1-1)} , T, \cdots, T,\Id^{  (2n+1-i_{k})}  \big)\Big(\mathfrak{X}_1, \cdots,  \mathfrak{X}_{n}, x_{n+1}\Big) \\
       &-\sum_{k=0}^{2n}\lambda^{2n-k}\sum_{1\leqslant i_1<i_2<\cdots< i_{k}\leqslant 2n+1}(-1)^{n+1} T_M \circ(\rho\cdot f)\circ  \Big( (1,\cdots,2n+1)+(1,\cdots,2n+1)\circ (1,\cdots,2n-1)\Big) \circ\\
       & \  \ \ \ \ \ \ \ \ \ \ \Big(\Id^{(i_1-1)} ,  T , \Id^{ (i_2-i_1-1)} , T, \cdots, T,\Id^{  (2n+1-i_{k})}  \Big)\Big(\mathfrak{X}_1, \cdots,  \mathfrak{X}_{n}, x_{n+1}\Big)\\
          	&+ (-1)^{n+1}  \rho\big(T(y_n), T(x_{n+1})\big)\cdot f\Big(T(x_1)\wedge T(y_1), \cdots, T(x_{n-1})\wedge T(y_{n-1}), T(x_{n})\Big)\\
        &+ (-1)^{n+1}\rho \big(T(x_{n+1}), T(x_n)\big)\cdot f\Big(T(x_1)\wedge T(y_1), \cdots, T(x_{n-1})\wedge T(y_{n-1}), T(y_{n})\Big).
 	\end{align*}
 \end{small}  So we have $\partial^n\Phi^{n}=\Phi^{n+1}\delta^n$.
	\end{proof}

	\smallskip

\noindent
{{\bf Acknowledgments.} This work is supported in part by Natural Science Foundation of China (Grant Nos. 12071137, 12161013, 11971460) and  by   STCSM  (No. 22DZ2229014).

The authors are very grateful to the referee for his/her careful reading of this paper and for his most useful comments which improve much the presentation of this paper.

\smallskip
	
\noindent{\bf Data availability statement.} Data sharing not applicable to this article as no datasets were generated or analysed during
the current study.
\smallskip

\noindent{\bf Conflict of interest statement.} The authors declared that they have no conflicts of interest to this work.

\end{document}